\pdfoutput=1
\documentclass[11pt,a4paper]{article}
\usepackage[margin=2.6cm]{geometry}
\usepackage{amsmath,amssymb,amsthm}
\usepackage{xcolor}
\usepackage{booktabs}
\usepackage{tikz}
\usetikzlibrary{arrows.meta}
\usepackage{url}
\usepackage[authoryear,round]{natbib}
\newcommand{\defval}[2]{\expandafter\gdef\csname val@#1\endcsname{#2}}
\newcommand{\val}[1]{\ifcsname val@#1\endcsname\csname val@#1\endcsname\else\textcolor{red}{[\texttt{\detokenize{#1}}]}\fi}
\defval{inst_D_m}{112}
\defval{inst_V_ms}{10}
\defval{inst_Pmax_kW}{3000}
\defval{inst_vcutin_ms}{3}
\defval{inst_vr_ms}{10.34}
\defval{inst_kappa}{2.7155}
\defval{inst_PV_kW}{2715}
\defval{inst_alpha}{0.05324}
\defval{inst_beta}{0.00000}
\defval{inst_gamma}{0.11859}
\defval{inst_rho}{0.293}
\defval{inst_sep_D}{3}
\defval{inst_stencil}{5}
\defval{inst_Nmax_plus}{2}
\defval{ejorpub_n10_t_s}{77}
\defval{ejorpub_n12_gap_pct}{7.98}
\defval{ejorpub_n14_gap_pct}{8.49}
\defval{ejorpub_n16_gap_pct}{9.74}
\defval{ejorpub_n20_gap_pct}{9.84}
\defval{ejorpub_n24_gap_pct}{9.15}
\defval{cor2_n10_ejor_runs}{5}
\defval{cor2_n10_ejor_certified}{5}
\defval{cor2_n10_ejor_solve_s}{77.3}
\defval{cor2_n10_ejor_best_below_pct}{0}
\defval{cor2_n10_full_runs}{5}
\defval{cor2_n10_full_certified}{5}
\defval{cor2_n10_full_tree_s}{0.1624}
\defval{cor2_n10_full_wall_s}{8}
\defval{cor2_n10_full_nodes}{1}
\defval{cor2_n10_full_root_gap_pct}{0.0898}
\defval{cor2_n10_full_kept}{68}
\defval{cor2_n10_full_maxrss_gb}{0.2095}
\defval{cor2_n10_cut_tangent_tree_s}{28.8}
\defval{cor2_n10_cut_upwind_tree_s}{31.1}
\defval{cor2_n12_ejor_runs}{5}
\defval{cor2_n12_ejor_certified}{0}
\defval{cor2_n12_ejor_gap_pct}{8.0}
\defval{cor2_n12_ejor_found_opt}{1}
\defval{cor2_n12_ejor_best_below_pct}{0}
\defval{cor2_n12_full_runs}{5}
\defval{cor2_n12_full_certified}{5}
\defval{cor2_n12_full_tree_s}{0.7429}
\defval{cor2_n12_full_wall_s}{48}
\defval{cor2_n12_full_nodes}{1}
\defval{cor2_n12_full_root_gap_pct}{0.2845}
\defval{cor2_n12_full_kept}{3\,224}
\defval{cor2_n12_full_maxrss_gb}{1.3}
\defval{cor2_n12_cut_tangent_tree_s}{179}
\defval{cor2_n12_cut_tangent_nodes}{152}
\defval{cor2_n12_cut_upwind_tree_s}{178}
\defval{cor2_n12_cut_upwind_nodes}{312}
\defval{cor2_n14_ejor_runs}{5}
\defval{cor2_n14_ejor_certified}{0}
\defval{cor2_n14_ejor_gap_pct}{8.5}
\defval{cor2_n14_ejor_best_below_pct}{1.1}
\defval{cor2_n14_base_tree_s}{646}
\defval{cor2_n14_base_wall_s}{1\,045}
\defval{cor2_n14_base_nodes}{775}
\defval{cor2_n14_base_kept}{248\,359}
\defval{cor2_n14_base_maxrss_gb}{12.1}
\defval{cor2_n14_base_root_s}{247}
\defval{cor2_n14_loss_tree_s}{680}
\defval{cor2_n14_loss_wall_s}{1\,004}
\defval{cor2_n14_loss_nodes}{791}
\defval{cor2_n14_loss_patterns}{3\,035\,916}
\defval{cor2_n14_loss_kept}{271\,906}
\defval{cor2_n14_loss_maxrss_gb}{8.3}
\defval{cor2_n14_loss_root_s}{191}
\defval{cor2_n14_full_runs}{5}
\defval{cor2_n14_full_certified}{5}
\defval{cor2_n14_full_tree_s}{74.7}
\defval{cor2_n14_full_wall_s}{395}
\defval{cor2_n14_full_nodes}{666}
\defval{cor2_n14_full_root_gap_pct}{0.4335}
\defval{cor2_n14_full_patterns}{3\,035\,916}
\defval{cor2_n14_full_kept}{57\,285}
\defval{cor2_n14_full_maxrss_gb}{8.3}
\defval{cor2_n14_full_tables_s}{72}
\defval{cor2_n14_full_root_s}{191}
\defval{cor2_n14_loss_nz_reduction_pct}{37.4}
\defval{cor2_n16_base_runs}{5}
\defval{cor2_n16_base_certified}{0}
\defval{cor2_n16_base_inc}{43396.22}
\defval{cor2_n16_base_maxrss_gb}{100}
\defval{cor2_n16_full_runs}{5}
\defval{cor2_n16_full_certified}{5}
\defval{cor2_n16_full_exact}{43672.38}
\defval{cor2_n16_full_tree_s}{3\,030}
\defval{cor2_n16_full_wall_s}{4\,923}
\defval{cor2_n16_full_nodes}{1\,028}
\defval{cor2_n16_full_kept}{1\,052\,511}
\defval{cor2_n16_full_inc}{43416.50}
\defval{cor2_n16_full_refixed}{579\,550}
\defval{cor2_n16_full_maxrss_gb}{69.2}
\defval{cor2_n16_baseinc_runs}{5}
\defval{cor2_n16_baseinc_certified}{0}
\defval{cor2_n16_baseinc_kept}{9\,591\,433}
\defval{cor2_n16_baseinc_inc}{43416.50}
\defval{cor2_n16_baseinc_bound}{43839.75}
\defval{cor2_gap_exact_max_pct}{0.0093}
\defval{cor2_wc_runs}{60}
\defval{cor2_wc_certified}{50}
\defval{cor2_probe_5026944_0_maxrss_gb}{14.3}
\defval{honest_n16_p40_certified}{5}
\defval{honest_n16_p40_exact}{43688.27}
\defval{honest_n16_p40_kept_final}{435\,543}
\defval{honest_n16_p40_maxrss_gb}{76.4}
\defval{honest_n16_p40_nodes}{1\,064}
\defval{honest_n16_p40_root_gap_pct}{0.5476}
\defval{honest_n16_p40_runs}{5}
\defval{honest_n16_p40_starts}{15}
\defval{honest_n16_p40_starts_at_exact}{8}
\defval{honest_n16_p40_tree_s}{1\,013}
\defval{honest_n16_p40_wall_max_s}{2\,658}
\defval{honest_n16_p40_wall_min_s}{2\,366}
\defval{honest_n16_p40_work}{4\,891}
\defval{honest_n16_p60_certified}{5}
\defval{honest_n16_p60_exact}{43704.15}
\defval{honest_n16_p60_kept_final}{874\,224}
\defval{honest_n16_p60_maxrss_gb}{76.6}
\defval{honest_n16_p60_nodes}{2\,549}
\defval{honest_n16_p60_root_gap_pct}{0.7471}
\defval{honest_n16_p60_runs}{5}
\defval{honest_n16_p60_starts}{15}
\defval{honest_n16_p60_starts_at_exact}{10}
\defval{honest_n16_p60_tree_s}{5\,081}
\defval{honest_n16_p60_wall_max_s}{7\,586}
\defval{honest_n16_p60_wall_min_s}{5\,860}
\defval{honest_n16_p60_wall_s}{6\,484}
\defval{honest_n16_p60_work}{27\,857}
\defval{honest_n16_uniform_certified}{5}
\defval{honest_n16_uniform_cpu_h}{1.6}
\defval{honest_n16_uniform_exact}{43672.38}
\defval{honest_n16_uniform_gap_exact_pct}{0.0004}
\defval{honest_n16_uniform_heur_s}{146}
\defval{honest_n16_uniform_kept}{3\,059\,875}
\defval{honest_n16_uniform_kept_final}{290\,489}
\defval{honest_n16_uniform_maxrss_gb}{59.1}
\defval{honest_n16_uniform_nodes}{742}
\defval{honest_n16_uniform_patterns}{18\,798\,206}
\defval{honest_n16_uniform_root_gap_pct}{0.4476}
\defval{honest_n16_uniform_root_work}{1\,393}
\defval{honest_n16_uniform_runs}{5}
\defval{honest_n16_uniform_starts}{15}
\defval{honest_n16_uniform_starts_at_exact}{12}
\defval{honest_n16_uniform_tree_s}{329}
\defval{honest_n16_uniform_wall_max_s}{1\,613}
\defval{honest_n16_uniform_wall_min_s}{1\,174}
\defval{honest_n16_uniform_wall_s}{1\,413}
\defval{honest_n16_uniform_work}{1\,553}
\defval{honest_n16_uniform_gap0_certified}{5}
\defval{honest_n16_uniform_gap0_gap_exact_pct}{0.0004}
\defval{honest_n16_uniform_gap0_kept_final}{290\,489}
\defval{honest_n16_uniform_gap0_maxrss_gb}{59.1}
\defval{honest_n16_uniform_gap0_nodes}{742}
\defval{honest_n16_uniform_gap0_runs}{5}
\defval{honest_n16_uniform_gap0_starts}{15}
\defval{honest_n16_uniform_gap0_starts_at_exact}{12}
\defval{honest_n16_uniform_gap0_tree_s}{319}
\defval{honest_n16_uniform_gap0_wall_max_s}{1\,371}
\defval{honest_n16_uniform_gap0_wall_min_s}{1\,248}
\defval{honest_n16_uniform_gap0_work}{1\,553}
\defval{honest_n17_uniform_gap_min_pct}{0.3508}
\defval{honest_n17_uniform_heur_best}{46236.01}
\defval{honest_n17_uniform_heur_s}{294}
\defval{honest_n17_uniform_kept_final}{1\,292\,378}
\defval{honest_n17_uniform_maxrss_gb}{160}
\defval{honest_n17_uniform_patterns}{49\,524\,384}
\defval{honest_n17_uniform_root_gap_pct}{0.7081}
\defval{honest_n17_uniform_root_s}{3\,343}
\defval{honest_n17_uniform_starts}{9}
\defval{honest_n17_uniform_starts_at_exact}{5}
\defval{honest_ref_cor2full_inc_gap_pct}{0.5859}
\defval{honest_ref_cor2full_kept_ratio}{3.6}
\defval{honest_ref_cor2full_root_work}{1\,446}
\defval{honest_ref_cor2full_work}{10\,345}
\defval{honest_ref_cor2full_work_ratio}{6.7}
\defval{tiles14_v4s2h4s2_regions}{12}
\defval{tiles14inc_v3h3_runs}{5}
\defval{tiles14inc_v3h3_certified}{5}
\defval{tiles14inc_v3h3_root_gap_pct}{0.77}
\defval{tiles14inc_v3h3_patterns}{264\,120}
\defval{tiles14inc_v3h3_kept}{155\,731}
\defval{tiles14inc_v3h3_kept_final}{57\,667}
\defval{tiles14inc_v3h3_nodes}{5\,202}
\defval{tiles14inc_v3h3_tree_s}{738}
\defval{tiles14inc_v3h3_work}{1\,582}
\defval{tiles14inc_v3h3_wall_s}{751}
\defval{tiles14inc_v3h3_maxrss_gb}{1.9}
\defval{tiles14inc_v4s2h4s2_runs}{5}
\defval{tiles14inc_v4s2h4s2_certified}{5}
\defval{tiles14inc_v4s2h4s2_root_gap_pct}{0.43}
\defval{tiles14inc_v4s2h4s2_patterns}{3\,035\,916}
\defval{tiles14inc_v4s2h4s2_kept}{276\,449}
\defval{tiles14inc_v4s2h4s2_kept_final}{58\,026}
\defval{tiles14inc_v4s2h4s2_nodes}{662}
\defval{tiles14inc_v4s2h4s2_tree_s}{45.4}
\defval{tiles14inc_v4s2h4s2_work}{134}
\defval{tiles14inc_v4s2h4s2_wall_s}{187}
\defval{tiles14inc_v4s2h4s2_maxrss_gb}{9.0}
\defval{tiles14inc_v5s3h5s3_runs}{5}
\defval{tiles14inc_v5s3h5s3_certified}{5}
\defval{tiles14inc_v5s3h5s3_root_gap_pct}{0.36}
\defval{tiles14inc_v5s3h5s3_patterns_ratio}{9.6}
\defval{tiles14inc_v5s3h5s3_kept}{458\,177}
\defval{tiles14inc_v5s3h5s3_kept_final}{94\,248}
\defval{tiles14inc_v5s3h5s3_root_s}{2\,084}
\defval{tiles14inc_v5s3h5s3_nodes}{227}
\defval{tiles14inc_v5s3h5s3_tree_s}{46.8}
\defval{tiles14inc_v5s3h5s3_work}{139}
\defval{tiles14inc_v5s3h5s3_wall_s}{2\,281}
\defval{tiles14inc_v5s3h5s3_maxrss_gb}{95.8}
\defval{tiles16inc_v3h3_runs}{3}
\defval{tiles16inc_v3h3_certified}{1}
\defval{tiles16inc_v3h3_root_gap_pct}{0.85}
\defval{tiles16inc_v3h3_nodes_min}{3\,620}
\defval{tiles16inc_v3h3_nodes_max}{13\,772}
\defval{tiles17inc_v4s2h4s2_nodes_min}{2\,250}
\defval{tiles17inc_v4s2h4s2_nodes_max}{3\,417}
\defval{tiles17inc_v4s2h4s2_tree_s_min}{39\,676}
\defval{tiles17inc_v4s2h4s2_tree_s_max}{42\,765}
\defval{tiles17inc_v4s2h4s2_gap_min_pct}{0.37}
\defval{tiles17inc_v4s2h4s2_gap_max_pct}{0.38}
\defval{tiles17inc_v4s2h4s2_open_nodes_min}{822}
\defval{tiles17inc_v4s2h4s2_open_nodes_max}{1\,008}
\defval{rose_runs}{78}
\defval{rose_solved}{78}
\defval{rose_d4_images}{17}
\defval{rose_rows}{18}
\defval{rose_nu_list}{0.35/0.15/0.35/0.15, 0.4/0.2/0.2/0.2, 0.4/0.3/0.2/0.1, 0.45/0.05/0.45/0.05, 0.6/0.25/0.1/0.05, 0.7/0.1/0.1/0.1}
\defval{rose_n10_uniform_nodes}{1}
\defval{rose_n10_uniform_tree_s}{0.20}
\defval{rose_n10_uniform_turbines}{11}
\defval{rose_n10_nu_runs}{30}
\defval{rose_n10_nu_solved}{30}
\defval{rose_n10_nu_nodes_min}{1}
\defval{rose_n10_nu_nodes_max}{47}
\defval{rose_n10_nu_tree_min_s}{0.00}
\defval{rose_n10_nu_tree_max_s}{7.8}
\defval{rose_n10_nu_root_gap_min_pct}{0.09}
\defval{rose_n10_nu_root_gap_max_pct}{0.89}
\defval{rose_n10_nu_d4_images}{5}
\defval{rose_n10_nu_roses}{6}
\defval{rose_n10_axial_turbines}{12}
\defval{rose_n12_uniform_nodes}{1}
\defval{rose_n12_uniform_tree_s}{12.1}
\defval{rose_n12_nu_runs}{30}
\defval{rose_n12_nu_solved}{30}
\defval{rose_n12_nu_nodes_min}{1}
\defval{rose_n12_nu_nodes_max}{1}
\defval{rose_n12_nu_tree_min_s}{0.30}
\defval{rose_n12_nu_tree_max_s}{0.70}
\defval{rose_n12_nu_root_gap_min_pct}{0.14}
\defval{rose_n12_nu_root_gap_max_pct}{0.17}
\defval{rose_n12_nu_d4_images}{6}
\defval{rose_n12_nu_roses}{6}
\defval{rose_n14_uniform_nodes}{833}
\defval{rose_n14_uniform_tree_s}{646}
\defval{rose_n14_nu_runs}{18}
\defval{rose_n14_nu_solved}{18}
\defval{rose_n14_nu_nodes_min}{31}
\defval{rose_n14_nu_nodes_max}{194}
\defval{rose_n14_nu_tree_min_s}{88.4}
\defval{rose_n14_nu_tree_max_s}{401}
\defval{rose_n14_nu_root_gap_min_pct}{0.19}
\defval{rose_n14_nu_root_gap_max_pct}{0.39}
\defval{rose_n14_nu_d4_images}{6}
\defval{rose_n14_nu_roses}{6}
\defval{cutin_layouts}{153}
\defval{cutin_grids}{10, 12, 14, 16, 17}
\defval{cutin_min_speed_ms}{3.91}
\defval{cutin_cells_below}{0}
\defval{cutin_vcutin_ms}{3}
\defval{cutin_d12_layouts}{8}
\defval{shift_certified}{221}
\defval{shift_start_runs}{36}
\defval{shift_tree_runs}{185}
\defval{shift_eps}{10^{-6}}
\defval{shift_unshifted_vs_ref_max_kw}{9.4e-10}
\defval{shift_resid_min_pct}{0.0002}
\defval{shift_resid_max_pct}{0.0004}
\defval{shift_resid_max_pct_up}{0.00042}
\defval{shift_shift_min_pct}{0.0002}
\defval{shift_shift_max_pct}{0.0004}
\defval{shift_obj_minus_shift_max_kw}{1.1e-02}
\defval{panel_instances}{36}
\defval{panel_both_certified}{38}
\defval{panel_pairs_same_layout}{36}
\defval{panel_pairs_other_layout}{2}
\defval{panel_full_only}{68}
\defval{panel_ejor_only}{0}
\defval{panel_neither}{2}
\defval{panel_agree_max_pct}{0}
\defval{panel_forbidden_min}{10}
\defval{panel_forbidden_max}{34}
\defval{panel_full_runs}{108}
\defval{panel_full_certified}{106}
\defval{panel_full_censored}{2}
\defval{panel_full_host}{razor}
\defval{panel_full_root_gap_pct}{0.338}
\defval{panel_full_kept_final_pct}{3.16}
\defval{panel_full_censored_gap_max_pct}{0.443}
\defval{panel_full_below_best_max_pct}{0.187}
\defval{panel_full_root_gap_pct_max}{1.06}
\defval{panel_full_n10_runs}{36}
\defval{panel_full_n10_certified}{36}
\defval{panel_full_n10_root_gap_pct}{0.202}
\defval{panel_full_n10_kept_final_pct}{2.55}
\defval{panel_full_n10_nodes}{1}
\defval{panel_full_n10_wall_s}{4}
\defval{panel_full_n10_gap_pct}{0.0061}
\defval{panel_full_n10_below_best_max_pct}{0.0000}
\defval{panel_full_n10_root_gap_pct_max}{1.06}
\defval{panel_full_n12_runs}{36}
\defval{panel_full_n12_certified}{36}
\defval{panel_full_n12_root_gap_pct}{0.342}
\defval{panel_full_n12_kept_final_pct}{4.29}
\defval{panel_full_n12_nodes}{285}
\defval{panel_full_n12_wall_s}{40.5}
\defval{panel_full_n12_gap_pct}{0.0008}
\defval{panel_full_n12_below_best_max_pct}{0.0000}
\defval{panel_full_n12_root_gap_pct_max}{0.720}
\defval{panel_full_n14_runs}{36}
\defval{panel_full_n14_certified}{34}
\defval{panel_full_n14_root_gap_pct}{0.580}
\defval{panel_full_n14_kept_final_pct}{3.47}
\defval{panel_full_n14_nodes}{846}
\defval{panel_full_n14_wall_s}{397}
\defval{panel_full_n14_gap_pct}{0.0009}
\defval{panel_full_n14_below_best_max_pct}{0.187}
\defval{panel_full_n14_root_gap_pct_max}{0.903}
\defval{panel_full_bm2_runs}{54}
\defval{panel_full_bm2_certified}{52}
\defval{panel_full_bm2_root_gap_pct}{0.583}
\defval{panel_full_bm2_kept_final_pct}{8.99}
\defval{panel_full_bm2_nodes}{858}
\defval{panel_full_bm2_wall_s}{65}
\defval{panel_full_bm2_gap_pct}{0.0002}
\defval{panel_full_bm2_below_best_max_pct}{0.187}
\defval{panel_full_bm2_root_gap_pct_max}{0.903}
\defval{panel_full_bp2_runs}{54}
\defval{panel_full_bp2_certified}{54}
\defval{panel_full_bp2_root_gap_pct}{0.140}
\defval{panel_full_bp2_kept_final_pct}{0.127}
\defval{panel_full_bp2_nodes}{1}
\defval{panel_full_bp2_wall_s}{30}
\defval{panel_full_bp2_gap_pct}{0.0063}
\defval{panel_full_bp2_below_best_max_pct}{0.0000}
\defval{panel_full_bp2_root_gap_pct_max}{1.06}
\defval{panel_full_uniform_runs}{54}
\defval{panel_full_uniform_certified}{54}
\defval{panel_full_uniform_root_gap_pct}{0.267}
\defval{panel_full_uniform_kept_final_pct}{1.98}
\defval{panel_full_uniform_nodes}{17}
\defval{panel_full_uniform_wall_s}{30.5}
\defval{panel_full_uniform_gap_pct}{0.0004}
\defval{panel_full_uniform_below_best_max_pct}{0.0000}
\defval{panel_full_uniform_root_gap_pct_max}{0.678}
\defval{panel_full_p40_runs}{54}
\defval{panel_full_p40_certified}{52}
\defval{panel_full_p40_root_gap_pct}{0.391}
\defval{panel_full_p40_kept_final_pct}{4.28}
\defval{panel_full_p40_nodes}{35.5}
\defval{panel_full_p40_wall_s}{56}
\defval{panel_full_p40_gap_pct}{0.0047}
\defval{panel_full_p40_below_best_max_pct}{0.187}
\defval{panel_full_p40_root_gap_pct_max}{1.06}
\defval{panel_full_lake_runs}{36}
\defval{panel_full_lake_certified}{36}
\defval{panel_full_lake_root_gap_pct}{0.183}
\defval{panel_full_lake_kept_final_pct}{0.928}
\defval{panel_full_lake_nodes}{1}
\defval{panel_full_lake_wall_s}{61.5}
\defval{panel_full_lake_gap_pct}{0.0062}
\defval{panel_full_lake_below_best_max_pct}{0.0000}
\defval{panel_full_lake_root_gap_pct_max}{0.790}
\defval{panel_full_road_runs}{36}
\defval{panel_full_road_certified}{36}
\defval{panel_full_road_root_gap_pct}{0.458}
\defval{panel_full_road_kept_final_pct}{5.32}
\defval{panel_full_road_nodes}{23.5}
\defval{panel_full_road_wall_s}{36.5}
\defval{panel_full_road_gap_pct}{0.0004}
\defval{panel_full_road_below_best_max_pct}{0.0000}
\defval{panel_full_road_root_gap_pct_max}{1.06}
\defval{panel_full_holes_runs}{36}
\defval{panel_full_holes_certified}{34}
\defval{panel_full_holes_root_gap_pct}{0.454}
\defval{panel_full_holes_kept_final_pct}{3.53}
\defval{panel_full_holes_nodes}{180}
\defval{panel_full_holes_wall_s}{55}
\defval{panel_full_holes_gap_pct}{0.0004}
\defval{panel_full_holes_below_best_max_pct}{0.187}
\defval{panel_full_holes_root_gap_pct_max}{0.903}
\defval{panel_ejor_runs}{108}
\defval{panel_ejor_certified}{38}
\defval{panel_ejor_host}{arrow}
\defval{panel_ejor_root_gap_pct}{5.92}
\defval{panel_ejor_censored_gap_max_pct}{10.01}
\defval{panel_ejor_below_best_max_pct}{2.11}
\defval{panel_ejor_root_gap_pct_max}{17.31}
\defval{panel_ejor_n10_runs}{36}
\defval{panel_ejor_n10_certified}{36}
\defval{panel_ejor_n10_root_gap_pct}{8.42}
\defval{panel_ejor_n10_nodes}{37\,754}
\defval{panel_ejor_n10_wall_s}{64.5}
\defval{panel_ejor_n10_gap_pct}{0.0000}
\defval{panel_ejor_n10_below_best_max_pct}{0.0000}
\defval{panel_ejor_n10_root_gap_pct_max}{17.31}
\defval{panel_ejor_n12_runs}{36}
\defval{panel_ejor_n12_certified}{2}
\defval{panel_ejor_n12_root_gap_pct}{6.76}
\defval{panel_ejor_n12_nodes}{156\,182}
\defval{panel_ejor_n12_wall_s}{3\,601}
\defval{panel_ejor_n12_gap_pct}{3.46}
\defval{panel_ejor_n12_below_best_max_pct}{1.28}
\defval{panel_ejor_n12_root_gap_pct_max}{12.79}
\defval{panel_ejor_n14_runs}{36}
\defval{panel_ejor_n14_certified}{0}
\defval{panel_ejor_n14_root_gap_pct}{5.92}
\defval{panel_ejor_n14_nodes}{69\,926}
\defval{panel_ejor_n14_wall_s}{3\,601}
\defval{panel_ejor_n14_gap_pct}{5.85}
\defval{panel_ejor_n14_below_best_max_pct}{2.11}
\defval{panel_ejor_n14_root_gap_pct_max}{9.01}
\defval{panel_ejor_bm2_runs}{54}
\defval{panel_ejor_bm2_certified}{18}
\defval{panel_ejor_bm2_root_gap_pct}{3.32}
\defval{panel_ejor_bm2_nodes}{114\,856}
\defval{panel_ejor_bm2_wall_s}{3\,601}
\defval{panel_ejor_bm2_gap_pct}{3.43}
\defval{panel_ejor_bm2_below_best_max_pct}{0.595}
\defval{panel_ejor_bm2_root_gap_pct_max}{3.58}
\defval{panel_ejor_bp2_runs}{54}
\defval{panel_ejor_bp2_certified}{20}
\defval{panel_ejor_bp2_root_gap_pct}{10.80}
\defval{panel_ejor_bp2_nodes}{58\,884}
\defval{panel_ejor_bp2_wall_s}{3\,601}
\defval{panel_ejor_bp2_gap_pct}{7.42}
\defval{panel_ejor_bp2_below_best_max_pct}{2.11}
\defval{panel_ejor_bp2_root_gap_pct_max}{17.31}
\defval{panel_ejor_uniform_runs}{54}
\defval{panel_ejor_uniform_certified}{20}
\defval{panel_ejor_uniform_root_gap_pct}{5.98}
\defval{panel_ejor_uniform_nodes}{93\,344}
\defval{panel_ejor_uniform_wall_s}{3\,601}
\defval{panel_ejor_uniform_gap_pct}{3.48}
\defval{panel_ejor_uniform_below_best_max_pct}{2.11}
\defval{panel_ejor_uniform_root_gap_pct_max}{16.93}
\defval{panel_ejor_p40_runs}{54}
\defval{panel_ejor_p40_certified}{18}
\defval{panel_ejor_p40_root_gap_pct}{5.92}
\defval{panel_ejor_p40_nodes}{91\,811}
\defval{panel_ejor_p40_wall_s}{3\,601}
\defval{panel_ejor_p40_gap_pct}{3.45}
\defval{panel_ejor_p40_below_best_max_pct}{2.06}
\defval{panel_ejor_p40_root_gap_pct_max}{17.31}
\defval{panel_ejor_lake_runs}{36}
\defval{panel_ejor_lake_certified}{12}
\defval{panel_ejor_lake_root_gap_pct}{5.92}
\defval{panel_ejor_lake_nodes}{107\,594}
\defval{panel_ejor_lake_wall_s}{3\,601}
\defval{panel_ejor_lake_gap_pct}{3.71}
\defval{panel_ejor_lake_below_best_max_pct}{2.11}
\defval{panel_ejor_lake_root_gap_pct_max}{15.46}
\defval{panel_ejor_road_runs}{36}
\defval{panel_ejor_road_certified}{14}
\defval{panel_ejor_road_root_gap_pct}{5.99}
\defval{panel_ejor_road_nodes}{74\,968}
\defval{panel_ejor_road_wall_s}{3\,601}
\defval{panel_ejor_road_gap_pct}{3.10}
\defval{panel_ejor_road_below_best_max_pct}{2.06}
\defval{panel_ejor_road_root_gap_pct_max}{17.31}
\defval{panel_ejor_holes_runs}{36}
\defval{panel_ejor_holes_certified}{12}
\defval{panel_ejor_holes_root_gap_pct}{6.26}
\defval{panel_ejor_holes_nodes}{96\,348}
\defval{panel_ejor_holes_wall_s}{3\,601}
\defval{panel_ejor_holes_gap_pct}{3.65}
\defval{panel_ejor_holes_below_best_max_pct}{1.85}
\defval{panel_ejor_holes_root_gap_pct_max}{14.35}
\defval{bound_n12_ejorpub_bound_above_pct}{7.6}
\defval{bound_n14_ejorpub_bound_above_pct}{7.3}
\defval{ejor16_tl3600_runs}{5}
\defval{ejor16_tl3600_certified}{0}
\defval{ejor16_tl3600_gap_pct}{9.9}
\defval{ejor16_tl3600_nodes}{45\,787}
\defval{ejor16_tl3600_wall_s}{3\,602}
\defval{ejor16_tl10800_runs}{5}
\defval{ejor16_tl10800_certified}{0}
\defval{ejor16_tl10800_gap_pct}{9.7}
\defval{ejor16_tl10800_best_below_pct}{2.2}
\defval{ejor16_tl10800_nodes}{82\,510}
\defval{ejor16_tl10800_wall_s}{10\,802}
\defval{ejor16up_cutoff_runs}{5}
\defval{ejor16up_cutoff_certified}{0}
\defval{ejor16up_cutoff_gap_pct}{7.3}
\defval{ejor16up_cutoff_solver_gap_pct}{10.3}
\defval{ejor16up_cutoff_nodes}{121\,894}
\defval{ejor16up_cutoff_wall_s}{10\,802}
\defval{ejor16up_cutoff_bound}{46856.10}
\defval{ejor16up_cutoff_maxrss_gb}{4.1}
\defval{ejor16up_plain_runs}{5}
\defval{ejor16up_plain_certified}{0}
\defval{ejor16up_plain_solver_gap_pct}{10.3}
\defval{ejor16up_plain_nodes}{108\,985}
\defval{ejor16up_plain_wall_s}{10\,802}
\defval{ejor16up_plain_maxrss_gb}{4.4}
\defval{heurcpu_seeds}{3}
\defval{heurcpu_same}{3}
\defval{heurcpu_cpu_tl}{10\,800}
\defval{heurcpu_at_gpu_tl_opt}{2}
\defval{heurcpu_ratio_min}{65}
\defval{heurcpu_ratio_max}{66}
\defval{heurcpu_cpu_t_opt_min}{508}
\defval{heurcpu_cpu_t_opt_max}{1\,986}
\defval{heurcpu_gpu_t_opt_min}{16.4}
\defval{heurcpu_gpu_t_opt_max}{38.6}
\defval{d12_n10_runs}{5}
\defval{d12_n10_certified}{5}
\defval{d12_n10_root_gap_pct}{4.67}
\defval{d12_n10_kept_pct}{93.6}
\defval{d12_n10_kept_final_pct}{89.1}
\defval{d12_n10_nodes}{2\,700}
\defval{d12_n10_tree_s}{146}
\defval{d12_n10_gap_pct}{0.0003}
\defval{d12_n10_maxrss_gb}{1.4}
\defval{d12_n12_runs}{5}
\defval{d12_n12_certified}{0}
\defval{d12_n12_root_gap_pct}{6.63}
\defval{d12_n12_kept_pct}{100}
\defval{d12_n12_kept_final_pct}{100}
\defval{d12_n12_nodes}{3\,121}
\defval{d12_n12_tree_s}{10\,800}
\defval{d12_n12_gap_pct}{5.93}
\defval{d12_n12_maxrss_gb}{11.1}
\defval{d12_n14_runs}{3}
\defval{d12_n14_certified}{0}
\defval{d12_n14_root_gap_pct}{9.50}
\defval{d12_n14_kept_pct}{100}
\defval{d12_n14_kept_final_pct}{100}
\defval{d12_n14_nodes}{7}
\defval{d12_n14_tree_s}{10\,802}
\defval{d12_n14_gap_pct}{9.43}
\defval{d12_n14_maxrss_gb}{43.3}
\defval{d12_ref4_n10_runs}{5}
\defval{d12_ref4_n10_certified}{5}
\defval{d12_ref4_n10_kept_pct}{0.78}
\defval{d12_ref4_n10_kept_final_pct}{0.10}
\defval{d12_ref4_n10_tree_s}{0.16}
\defval{d12_ref4_n10_nodes}{1}
\defval{d12_ref4_n10_root_gap_pct}{0.090}
\defval{d12_ref4_n10_gap_pct}{0.0003}
\defval{d12_ref4_n10_maxrss_gb}{0.21}
\defval{d12_ref4_n12_runs}{5}
\defval{d12_ref4_n12_certified}{5}
\defval{d12_ref4_n12_kept_pct}{4.0}
\defval{d12_ref4_n12_kept_final_pct}{0.68}
\defval{d12_ref4_n12_tree_s}{0.74}
\defval{d12_ref4_n12_nodes}{1}
\defval{d12_ref4_n12_root_gap_pct}{0.285}
\defval{d12_ref4_n12_gap_pct}{0.0004}
\defval{d12_ref4_n12_maxrss_gb}{1.3}
\defval{d12_ref4_n14_runs}{5}
\defval{d12_ref4_n14_certified}{5}
\defval{d12_ref4_n14_kept_pct}{9.0}
\defval{d12_ref4_n14_kept_final_pct}{1.9}
\defval{d12_ref4_n14_tree_s}{74.7}
\defval{d12_ref4_n14_nodes}{666}
\defval{d12_ref4_n14_root_gap_pct}{0.433}
\defval{d12_ref4_n14_gap_pct}{0.0004}
\defval{d12_ref4_n14_maxrss_gb}{8.3}
\defval{heur_width}{4}
\defval{heur_win_rows}{8}
\defval{heur_win_cols}{8}
\defval{heur_stride}{2}
\defval{heur_tl}{1\,500}
\defval{heur_topk_w}{eight}
\defval{heur_kicks}{100}
\defval{heur_kick_w}{four}
\defval{heur_patience_w}{thirty}
\defval{heur_rounds_w}{twenty}
\defval{heur_starts_w}{three}
\defval{heur_gputol}{0.05}
\defval{tol_sep}{$10^{-6}$}
\defval{tol_fix}{$10^{-6}$}

\newtheorem{proposition}{Proposition}
\newtheorem{lemma}{Lemma}
\newcommand{\Nmax}{N_{\max}}

\title{A Dantzig--Wolfe reformulation of the EndoWake model for optimal wind farm layouts}
\author{Martina Fischetti\thanks{Instituto de Matem\'aticas de la Universidad de Sevilla
  (IMUS), Universidad de Sevilla, Seville, Spain. ORCID 0000-0002-7673-6917. Corresponding
  author: mfischetti@us.es.}
  \and Matteo Fischetti\thanks{Department of Information Engineering, University of Padova,
  Via Gradenigo 6/B, 35131 Padova, Italy. ORCID 0000-0001-6601-0568. Email:
  matteo.fischetti@unipd.it.}}
\date{October 2026}

\begin{document}
\maketitle

\begin{abstract}
Wind Farm Layout Optimization places turbines to maximize
the expected energy under a wind rose, each turbine lowering the wind downwind. In
the EndoWake model the wind field is a variable, propagated
by linear inequalities, and layouts are certified by branch-and-check.
On a ladder of grids it certifies the $10\times10$ grid and leaves the larger grids
open after one hour. We tighten the relaxation by a convexification computed rather than
derived: every pattern of turbines inside a small tile of the grid is
enumerated once with its power and enters the master as a column.
The device is valid because the wind field is antitone, and exact, for the cubic
power curve, with a Benders cut at integer layouts. Reduced-cost fixing and propagation
between overlapping tiles shrink the master before the tree search, and the tree repeats the
fixing at every improving incumbent. The incumbent comes from a tile heuristic or an LP dive. On the same computer as the branch-and-check,
the $12\times12$ and $14\times14$ grids close in
\val{cor2_n12_full_wall_s} and \val{cor2_n14_full_wall_s} seconds. The $16\times16$ grid is
certified on five seeds from random starts, one job per seed, the heuristic on a GPU, in
\val{honest_n16_uniform_wall_s} seconds under the uniform rose and
\val{honest_n16_p60_wall_s} seconds under a non-uniform one. Times are medians. A pre-specified test set of
\val{panel_instances} unseen instances with forbidden cells confirms the advantage in the
axial setting. Certified means
closed at the solver's gap of $10^{-4}$: the optimum lies within it plus a residual
of at most \val{shift_resid_max_pct_up}\,\%.
\end{abstract}


\section{Introduction}\label{sec:intro}

The \emph{Wind Farm Layout Optimization} (WFLO) problem asks where to place up to a
prescribed number of turbines inside an admissible region so as to maximize the expected energy production
under a wind rose; see \citet{fischetti2019overview} for an overview of its many variants, and
\citet{fischetti2026endowake} for an extensive review of the literature.
The difficulty is the wake: a turbine downwind of another sees a reduced wind speed. Since
power grows with the cube of the speed in the sub-rated regime, the energy of a layout
is a non-separable set function of the chosen sites. The classical mathematical-programming
approaches discretize the region into candidate sites and precompute, by an analytical wake
model, how each candidate turbine slows down the
others~\citep{donovan2005,archer2011,turner2014,fischetti2016proximity,ulku2019}. On these
pairwise coefficients, \citet{senderovich2022} build quadratic and QUBO formulations, and
\citet{pedersen2026} optimize placement and cable routing together in an exact framework.
Other models go beyond pairwise losses: \citet{zhang2014} capture the nonlinear interaction
of multiple wakes, and \citet{perezrua2023} give a wind-speed variable to every candidate
site, with a stepwise power curve. \citet{kuo2016} update the wake
coefficients by CFD simulations between successive MIPs, outside the formulation.
The \emph{EndoWake} model, which we introduced in \citet{fischetti2026endowake}, makes the
flow field itself a variable: the wind speed at every cell of a grid, empty or not, is a
decision variable, linked to its upwind neighbors by linear propagation inequalities. The same mechanism then yields the linear relaxation,
the valid inequalities and the exact evaluation of a layout. The layouts that EndoWake
certifies are optimal on the exact energy of the model, which is the quantity we shall
always refer to.

Unfortunately, certifying a layout turned out to be hard. The linear relaxation of the
model as written carries no information, and the band inequalities proposed with
it repair it only in part. The EndoWake paper solves the model by
branch-and-check: one hypograph variable per cell and wind direction, integer layouts
evaluated by an exact linear-time oracle, and problem-specific upwind cuts as optimality
cuts. Its ladder is the square grids from $10\times10$ to $16\times16$, in steps of two. The
scheme certifies the $10\times10$ grid in \val{ejorpub_n10_t_s} seconds and none of the
larger ones. After one hour, in the runs published there, the $12\times12$ and the $14\times14$ grids stand at gaps of
\val{ejorpub_n12_gap_pct}\,\% and \val{ejorpub_n14_gap_pct}\,\%, and the $16\times16$ at
\val{ejorpub_n16_gap_pct}\,\%. Note that good layouts are not the issue, as the EndoWake
paper observes: a greedy construction followed by local search produces them in seconds, and the
tile heuristic of Section~\ref{sec:heur} reaches the optimum of the $16\times16$ grid. The
issue is the dual bound, which is what a certificate adds to a good layout: for a fixed
model, it bounds the energy that any further search can gain, and it is the reference
against which a heuristic is judged. A good incumbent still counts: it is the value
against which the reduced-cost fixing of Section~\ref{sec:shrink} measures the bound. Given
the optimum as cutoff, the branch-and-check still closes no seed of the $16\times16$ grid in
three hours, and its bound stays \val{ejor16up_cutoff_gap_pct}\,\% above the optimum
(Section~\ref{sec:exp:16}).

We keep the EndoWake model, its instances, its
exact oracle and the frame of its algorithm (a master over the layout and the powers, the
incumbent entered at its exact energy, lazy optimality cuts at integer points). What we
replace is the relaxation and the search, and the branch-and-check is the baseline of
Section~\ref{sec:exp}, run as published. In this paper we investigate how far a certificate can be pushed,
and at what cost, when the relaxation is tightened by a convexification \emph{precomputed} on
small tiles of the grid and the master problem is shrunk before the tree search. The device
works as follows. A tile is a strip
of a few rows (or columns) of the grid. The layouts that fit inside a tile, that we call its
\emph{patterns}, are few enough to be enumerated, and the power of every cell of the
tile under every pattern is tabulated by one batched propagation of the wind field. We then
let the patterns enter the master problem as columns, in the spirit of Dantzig--Wolfe
decomposition: the layout of every tile is a convex combination of its patterns. The
power of each cell of the tile is bounded by the same combination of the pattern powers.
The inequality is valid because the wind field of EndoWake is \emph{antitone}, i.e.,
turbines outside the tile can only lower the power of its cells. The model stays exact
once the convexified power is paired with a Benders cut at integer layouts, which we derive
from the concavity of the projected field in a suitable coordinate. The band inequalities of
EndoWake turn out to be a relaxation of the one-dimensional case of the device, with
the pattern power aggregated over a band.

The coordinate in which the cut is derived is the second idea of the paper. It removes the
power curve from the model and gives it a concave second stage; its gain is in the model
rather than in the running time (Section~\ref{sec:exp:attr}). EndoWake
propagates the wind \emph{speed}, and the power, cubic in the speed below rated, enters the
original model as a piecewise-linear interpolant under SOS2
conditions. The solver then branches on the curve, the field, as a second-stage problem, is
non-concave, and the EndoWake paper finds no LP dual multipliers for a Benders decomposition.
Instead, we propagate the \emph{cube} of the speed, i.e., the power itself. The propagation
of EndoWake is explicit and monotone, so the change of coordinate leaves the planes, the
stencil, the oracle, antitonicity and the upwind cut untouched. The power becomes linear,
with no interpolant and no SOS2 condition, the flow inequality becomes a concave power mean
of order $1/3$, and the field projected on the layout is a concave function of the layout.
A supergradient costs one reverse sweep, the chain rule walked backwards through the planes
as in backpropagation, and the cut at integer layouts is the optimality cut of generalized
Benders decomposition. What is new, to the best of our knowledge, is the use of the cube of
the speed, the natural coordinate of a cubic power curve, to obtain a concave extension of
EndoWake, whose supergradient cuts come from one reverse sweep. The construction rests on a
wind field that is a variable of the model, propagated by inequalities that are concave in
that coordinate.

The columns are many (\val{honest_n16_uniform_patterns} on the $16\times16$ grid), hence
the second half of the paper is about not branching on them. Given an incumbent, the reduced
costs of the root LP remove the columns that cannot enter a better layout. A propagation
between overlapping tiles removes the survivors that have no compatible partner on the shared
cells, once, after the fixing at the root; the reduced-cost test alone is re-applied in the
search tree at every improving incumbent. The
incumbent itself comes from a tile heuristic that evaluates, on a GPU, every pattern of a
tile with the rest of the layout fixed (the GPU evaluates EndoWake layouts for design in
\citet{fischetti2026gpu}; here it serves the certificate). With \val{heur_starts_w} random
starts per run it reached the optimum on all five seeds of the three $16\times16$ roses
(Section~\ref{sec:heur}). An ablation on the ladder attributes the gain to its levers: the sparse
form of the master saves root time and memory, the propagation between tiles shortens the tree.
As to scope, the tiles are tight on the axial wakes only. The pipeline is measured on one
instance family: the grids from $10\times10$ to $16\times16$ of the EndoWake ladder
under uniform and non-uniform axial roses (the fixing alone on
six of them). A test set fixed before any run adds \val{panel_instances} unseen instances up to
$14\times14$, with forbidden cells, two budgets and a non-uniform rose
(Section~\ref{sec:exp:panel}). On a rose of twelve directions (Table~\ref{tab:d12}) it closes the
$10\times10$ grid and not the $12\times12$: the root bound of the oblique propagation is
loose, and Section~\ref{sec:limits} says why.

Each ingredient has a precedent, which we credit below; what is new is how the structure of
EndoWake makes them work together. Antitonicity makes the power inequalities of the tiles
valid, and the cube coordinate gives the concave second stage whose Benders cut recovers,
at integer layouts, the wakes that cross the borders of the tiles. The resulting relaxation is tight enough for the enumeration and reduction of
the columns to close grids that the branch-and-check leaves open. Pattern columns
are the cutting-stock columns of \citet{gilmore1961}, a Dantzig--Wolfe decomposition
\citep{dantzig1960} applied to small blocks; here the blocks overlap and are glued by the
site variables, as in the arrowhead form of Dantzig--Wolfe reformulation, where the linking
variables stay in the master \citep{bergner2015}. With one copy of $x$ per block the same
master is the Lagrangean decomposition of \citet{guignard1987}. Each block is a relaxation
of the model, exact only together with a cut at integer points. Computing the hull of a
small substructure by enumeration is also the idea of local cuts
\citep{chvatal2013,buchheim2008}, which separate its facets; we keep the hull in extended
form, so that reduced costs can fix its columns.
Enumerating the feasible configurations of a few cells of a grid, with clique inequalities
on those that cannot coexist, is how \citet{goycoolea2005} formulate harvest scheduling
under maximum-area restrictions; our columns carry a power inequality per cell and
direction, not a packing constraint. Enumerating the Dantzig--Wolfe columns, discarding
those whose reduced cost exceeds the gap to an incumbent and handing the survivors to a MIP
solver is the route enumeration of exact vehicle routing. It is the set-partitioning line of
Mingozzi and coauthors since \citet{bianco1994}, in the form of \citet{baldacci2008} and
\citet{baldacci2011}. \citet{irnich2010} eliminate arcs by the reduced costs of paths,
\citet{focacci1999} filter domains by the reduced costs of optimization relaxations, and
\citet{yang2025} fixes columns by Lagrangian underestimates. What the tiles add is the propagation between overlapping blocks, a device
the routes do not have: two routes of a solution share no customer, two tiles share their
overlap. Overlapping blocks are the setting of the consistency cuts of \citet{clausen2022},
equations that match the column mass of two blocks on each assignment of their shared
variables; our propagation is their support form, applied once after the fixing
(Section~\ref{sec:shrink}). The refixing at improving incumbents is standard reduced-cost fixing, with the root
reduced costs and the value of the new incumbent; \citet{yang2025} develops stronger forms of
repeated fixing, based on Lagrangian underestimates and on the choice of the dual solution. The Benders cut is the optimality cut of
generalized Benders decomposition \citep{geoffrion1972}, an outer approximation of a
concave value function in the sense of \citet{duran1986}, and the cube coordinate is what
makes the value concave. The upwind cut of the branch-and-check is its logic-based
form \citep{hooker2003}.

Our main contributions are the following.
\begin{itemize}
\item A Dantzig--Wolfe reformulation of EndoWake on tiles. The patterns of each tile enter
  the master as columns, and their power inequalities, valid by antitonicity, are tabulated
  by one batched propagation of the wind field. The band inequalities of
  EndoWake are a relaxation of its one-dimensional case (Section~\ref{sec:tiles}).
\item The cube coordinate. The power becomes linear, with no SOS2 condition, and the field
  projected on the layout becomes concave. The Benders cut at integer layouts, computed by
  one reverse sweep, keeps the master exact for the unshifted cubic program of
  Section~\ref{sec:model:cube}; the implementation adds a small shift, whose residual we
  measure (Section~\ref{sec:exp:setup}).
\item A shrinking of the master: reduced-cost fixing at the root and at every improving
  incumbent, a propagation between overlapping tiles, and a tile heuristic on the GPU for
  the incumbent (Sections~\ref{sec:shrink} and~\ref{sec:heur}).
\item On the computers of the branch-and-check, the pipeline certifies every seed of the
  $12\times12$ and $14\times14$ grids, which that code leaves open after one hour. On a node
  with a GPU it certifies the $16\times16$ grid in a single job per seed, which the
  branch-and-check leaves open in three hours. A test set fixed before any run confirms the advantage on unseen instances with forbidden
  cells, in the axial setting (Section~\ref{sec:exp}).
\end{itemize}

The paper is organized as follows. Section~\ref{sec:model} recalls the EndoWake model and
derives the Benders cut. Section~\ref{sec:tiles} introduces the pattern columns on tiles,
Section~\ref{sec:shrink} the shrinking of the master and Section~\ref{sec:heur} the tile
heuristic on the GPU. Section~\ref{sec:exp} reports the experiments.
Section~\ref{sec:limits} discusses the scope and the limitations of the approach, and some
conclusions are finally
drawn in Section~\ref{sec:concl}.

\section{The EndoWake model and the Benders cut}\label{sec:model}

We state the instance and the model of \citet{fischetti2026endowake} in the form used
below, then rewrite the propagation in the cube of the wind speed, where the field is a
concave function of the layout, and derive the Benders cut from concavity.

\subsection{The instance and the model}\label{sec:model:inst}

\paragraph{The instance.} The site is an $n\times n$ grid of square cells of side
$\Delta = D$, the rotor diameter of the turbine ($D = \val{inst_D_m}$ m). Every cell may host
a turbine, two turbines must be at least $\val{inst_sep_D}D$ apart, and at most
$\Nmax = n + \val{inst_Nmax_plus}$ turbines are installed, at no cost. The wind rose has four
axial directions (north, east, south and west) with probabilities $p_s$, equal in the ladder
and unequal in Section~\ref{sec:exp}, all at the free-stream
speed $V = \val{inst_V_ms}$ m/s. The turbine has the ideal cubic power curve
$P(w) = \kappa w^3$ between the cut-in speed of \val{inst_vcutin_ms} m/s and the rated speed
of \val{inst_vr_ms} m/s, zero below cut-in and $P_{\max} = \val{inst_Pmax_kW}$ kW above
rated, with $\kappa = \val{inst_kappa}$ kW/(m/s)$^3$. Since $V$ is below the rated speed, the
plateau never binds, and $P(V) = \val{inst_PV_kW}$ kW is the power of an undisturbed
machine. The objective below is an expected power in kW; the energy over a period differs
from it by a constant factor, and we say energy for short. The wake kernel is the calibrated one of \citet{fischetti2026endowake}: a turbine
leaves the fraction $\rho = \val{inst_rho}$ of the speed to the cell directly downwind of it,
and the deficit spreads sideways and recovers with the three parameters
$(\alpha,\beta,\gamma) = (\val{inst_alpha}, \val{inst_beta}, \val{inst_gamma})$ that enter
the propagation below. We use the grids from $10\times10$ to $16\times16$ of the ladder of
\citet{fischetti2026endowake}, which runs from $6\times6$ to $24\times24$: its
branch-and-check leaves the $20\times20$ and $24\times24$ grids open at
\val{ejorpub_n20_gap_pct}\,\% and \val{ejorpub_n24_gap_pct}\,\%, and we do not attempt them
(Section~\ref{sec:limits}).

\paragraph{The model.} A binary $x_c$, $c \in \mathcal{C}$, selects the cells that receive a
turbine; for every wind direction $s$ a continuous $w^s_c \in [0,V]$ is the wind speed at $c$
and $z^s_c \in [0, P(V)]$ the power credited to $c$. For direction $s$ the cells are grouped
into propagation planes orthogonal to the wind and swept from upwind to downwind. Every cell
$c$ reads the \val{inst_stencil} cells $\mathrm{up}_s(c,m)$, $m = -2,\dots,2$, of the
preceding plane at transversal offsets $m$ (the free stream $V$ for those outside the grid),
and $\mathcal{F}_s(c)$ is the cell directly downwind of $c$. The model then reads as follows:
\begin{align}
\max\ & \textstyle\sum_{s} p_s \sum_{c} z^s_c \notag\\
\text{s.t.}\quad
& w^s_c \le \gamma V + (1-\gamma) \textstyle\sum_{m=-2}^{2} K_m\, w^s_{\mathrm{up}_s(c,m)},
& & c \in \mathcal{C},\ s \in \mathcal{S}, \label{eq:flow}\\
& w^s_{c'} \le \rho\, w^s_c + V\,(1-x_c),
& & c' \in \mathcal{F}_s(c),\ s \in \mathcal{S}, \label{eq:wake}\\
& z^s_c \le P(w^s_c), \quad z^s_c \le P(V)\, x_c,
& & c \in \mathcal{C},\ s \in \mathcal{S}, \label{eq:power}\\
& \textstyle\sum_{c \in Q} x_c \le 1 \ \ (Q \in \mathcal{Q}), \quad
  \textstyle\sum_{c} x_c \le \Nmax, \quad x_c \in \{0,1\}, \label{eq:layout}
\end{align}
where $K = (\beta/2, \alpha/2, 1-\alpha-\beta, \alpha/2, \beta/2)$ is non-negative and sums
to one, and $\mathcal{Q}$ is the family of $3\times3$ windows of the grid (two cells of a
window are closer than $3D$). The flow inequalities \eqref{eq:flow} cap the speed at a cell by
a convex combination of the free stream and of the speeds one plane upwind. The wake
inequalities \eqref{eq:wake} let an installed turbine hand the fraction $\rho$ of its speed to
the cell behind it, and are void otherwise. In the original model the power curve
of \eqref{eq:power} is piecewise linear with SOS2 conditions; here the curve is used exactly,
through the oracle below and through the cube coordinate of Section~\ref{sec:model:cube}.
The cap $P_{\max} x_c$ of the original model becomes $P(V)\,x_c$, the cap of its
master, valid since $V$ is below rated.

For a layout $x$ and a direction $s$, let $\Phi_x(w)$ be the vector of the smallest right-hand
sides that \eqref{eq:flow}--\eqref{eq:wake} impose on each cell, so that the feasible fields
are $\{w \in [0,V]^{\mathcal{C}} : w \le \Phi_x(w)\}$. Everything below rests on the following
property, restated from \citet[Prop.~1]{fischetti2026endowake} because the pattern columns
are valid by antitonicity, its part~(d) and part~(c) here.

\begin{proposition}[Maximal field]\label{prop:field}
For every $x \in [0,1]^{\mathcal{C}}$ and every direction $s$: (a) the feasible fields have a
greatest element $w^*(x)$, computed by one forward sweep of the planes with equality in
$w = \Phi_x(w)$; (b) every objective non-decreasing in $w$ attains its maximum over the
feasible fields at $w^*(x)$; (c) \emph{(antitonicity)} $x' \ge x$ componentwise implies
$w^*(x') \le w^*(x)$ componentwise.
\end{proposition}

\begin{proof}
$\Phi_x$ is monotone in $w$ (every right-hand side has non-negative coefficients on $w$) and
antitone in $x$ at fixed $w$ (the layout enters only through $V(1-x_c)$). Since a cell reads
only cells strictly upwind, the sweep is well defined and the induction along the planes gives
$w \le \Phi_x(w) \le \Phi_x(w^*) = w^*$ for every feasible $w$, which is (a). Part (b) is
immediate. For (c), at a cell $c$ with the claim holding upwind, $w^*(x')_c = \Phi_{x'}(w^*(x'))_c \le
\Phi_x(w^*(x'))_c \le \Phi_x(w^*(x))_c = w^*(x)_c$.
\end{proof}

Part (a) is also the evaluator: one sweep per direction, five multiply--add operations and one
comparison per cell, plus a power-curve lookup per installed cell, returns the exact energy of
an integer layout $X = \{c : x_c = 1\}$, namely $\sum_s p_s \sum_{c \in X} p^s_c(X)$ with
$p^s_c(X) = P(w^{*s}_c(X))$. Every energy quoted in this paper is recomputed by this evaluator,
cut-in included.

\paragraph{The master.} The branch-and-check scheme of EndoWake keeps
the layout variables $x$, the power variables $z$, the constraints \eqref{eq:layout}, the
caps $z^s_c \le P(V)\,x_c$ and a class of so-called \emph{band inequalities}
(Section~\ref{sec:tiles:col}), kept static, and drops the field. At every integer solution $\bar x$ of the
master with support $X$, the oracle computes $p^s_c(X)$ and every cell with
$\bar z^s_c > p^s_c(X)$ receives an optimality cut. The layout is posted to the solver at
its exact energy, hence pruning acts on the exact scale from both sides. The optimality cut
of the branch-and-check is the so-called \emph{upwind cut}
\begin{equation}\label{eq:upwind}
z^s_c \le p\, x_c + \bigl(P(V) - p\bigr) \textstyle\sum_{u \in U} (1 - x_u),
\end{equation}
with $U \subseteq X \setminus \{c\}$ and $p$ the exact power of $c$ when turbines stand on
$U \cup \{c\}$ and nowhere else. The cut is valid by antitonicity. Note that without the band
inequalities the master has the vacuous bound $P(V)$ per installed cell. Everything that tightens it (the band inequalities
and the pattern columns of Section~\ref{sec:tiles} alike) is
a valid inequality on $z$ in terms of $x$, and the exactness of the scheme rests on the cuts
added at integer points.

\subsection{The cube coordinate}\label{sec:model:cube}

We next rewrite the propagation in the cube of the speed, i.e., in the power itself.
Since the propagation is explicit and $w \mapsto w^3$ is increasing, the planes, the
stencil, the oracle and Proposition~\ref{prop:field} are untouched, and only the shape of
the inequalities changes. Fix a direction and drop its index. Put $q_c = (w_c/V)^3 \in [0,1]$. Dividing
\eqref{eq:flow} by $V$ and cubing gives
\begin{equation}\label{eq:mean}
q_c \le M_c(q) := \Bigl( a_0 + \textstyle\sum_m a_m\, q_{\mathrm{up}(c,m)}^{1/3} \Bigr)^{3},
\qquad a_m = (1-\gamma)K_m, \quad a_0 + \textstyle\sum_m a_m = 1,
\end{equation}
where $a_0$ is $\gamma$ plus the weights of the neighbors outside the grid. Note that $M_c$ is
the weighted power mean of order $1/3$ of the numbers $1$ and $q_{\mathrm{up}(c,m)}$.

\begin{lemma}\label{lem:mean}
$M_c$ is concave and non-decreasing on the non-negative orthant.
\end{lemma}

\begin{proof}
$h(t,q) = (a_0 t^{1/3} + \sum_m a_m q_m^{1/3})^3$ is non-negative and positively homogeneous
of degree one on the non-negative orthant, and its upper level sets are those of the concave
function $a_0 t^{1/3} + \sum_m a_m q_m^{1/3}$, so $h$ is quasi-concave. A non-negative,
degree-one homogeneous, quasi-concave function on a convex cone is concave, and $M_c$ is the
restriction of $h$ to $t = 1$. Monotonicity is clear.
\end{proof}

In the implementation $M_c$ is evaluated at $q + \varepsilon$, every coordinate shifted by
$\varepsilon = \val{shift_eps}$, so that its supergradient stays finite at $q = 0$. We write
$\tilde M_c$ for the shifted mean and $\tilde q(x)$ for the field it propagates. Since
$\tilde M_c \ge M_c$ and both are non-decreasing, $\tilde q(x) \ge q(x)$ at every layout, by
induction along the planes: the shift is a relaxation of the model, a parameter of the
implementation and not a tolerance. The tables of Section~\ref{sec:tiles} and the cuts of
Section~\ref{sec:model:cut} are computed on $\tilde q$, and every bound they prove is valid
for $q$. At an integer layout that the solver finds, the master credits cell $c$ a value close
to $\kappa V^3 \tilde q_c$. The single-source chords of Section~\ref{sec:model:cut}, computed
on the unshifted field, can cap it below that value, and the tolerances of the separation can
let it exceed it slightly. Section~\ref{sec:exp:setup} measures the credited value against
the exact power and against the shifted field.

At a binary $x_c$ the wake inequality \eqref{eq:wake} is equivalent to the linear inequality
$q_{c'} \le \rho^3 q_c + (1 - x_c)$, and the cubic power is linear in $q$: for $w \le V$ one
has $P(w) \le \kappa w^3 = \kappa V^3 q$, with equality above cut-in, so
$z_c \le \kappa V^3 q_c$ is valid together with $z_c \le P(V)\,x_c$. Hence EndoWake written
in $q$ is a mixed-integer convex program, which we call the \emph{cubic program}, with
concave propagation inequalities, linear
wake inequalities and linear power, and no piecewise-linear error. It is exact up to the
cut-in jump, which it drops: a cell below cut-in is credited $\kappa w^3$ instead of zero.
Every bound the cubic program proves is therefore valid for the true curve, and a certificate
transfers to the true curve as soon as no installed cell of the certified layout is below
cut-in in any direction. Note that good layouts tend to avoid this situation, since an
installed cell below cut-in produces no energy in that direction. On the layouts stored by the campaigns of
Sections~\ref{sec:exp} and~\ref{sec:limits} this does not happen. We evaluated
\val{cutin_layouts} distinct layouts on the grids $n\times n$, $n = \val{cutin_grids}$, with
the reference code of EndoWake: incumbents, optima and final layouts of
the tree, \val{cutin_d12_layouts} of them on the rose of twelve directions. The smallest
speed seen by an installed cell is \val{cutin_min_speed_ms} m/s, against a cut-in of
\val{cutin_vcutin_ms} m/s (\val{cutin_cells_below} cells below it). This is evidence on the
stored layouts, not a proof over all layouts; the separation, which keeps the nearest upwind
turbine three cells away, is what makes it plausible. The advantage of the coordinate is
threefold. There is no SOS2 condition and no branching on the curve; an integer optimum
certifies the energy of the model, up to the shift, and not an SOS2 surrogate; and the
second stage is concave, so that generalized Benders
decomposition applies (Section~\ref{sec:model:cut}). The cube is the smallest exponent that
works, and the only one with a linear power.
With $q = w^r$ the propagation is a power mean of order $1/r$, concave for $r \ge 1$, while
the power $\kappa q^{3/r}$ is concave only for $r \ge 3$, and linear only for $r = 3$. We do
not solve the cubic program as such, and not only because the field variables are its
cost. Its continuous relaxation is vacuous at the uniform point $\bar x_c = \Nmax/n^2$ of
the EndoWake paper (feasible for $n \ge 11$; at $10\times10$ the midpoint of two
disjoint feasible layouts does the same): every wake inequality then releases to
$q \ge 1 - \bar x_c \ge \bar x_c$, and $z_c = P(V)\,\bar x_c$ is feasible. A Benders cut
separated at fractional points of this relaxation, the cubic program without the pattern
columns of Section~\ref{sec:tiles}, cannot therefore do better than the trivial bound
$\Nmax P(V)$. We use the coordinate to project the
field out. Extend the sweep of Proposition~\ref{prop:field}(a) to fractional layouts by
setting, plane after plane and for every $x \in [0,1]^{\mathcal{C}}$,
\begin{equation}\label{eq:qproj}
q_c(x) = \min\Bigl\{ 1,\ M_c\bigl(q(x)\bigr),\
\min_{u : c \in \mathcal{F}(u)} \bigl[\rho^3 q_u(x) + 1 - x_u\bigr] \Bigr\}.
\end{equation}

\begin{proposition}\label{prop:concave}
For every cell $c$, the map $x \mapsto q_c(x)$ is concave and non-increasing on
$[0,1]^{\mathcal{C}}$, and at every binary $x$ it equals $(w^*_c(x)/V)^3$.
\end{proposition}

\begin{proof}
By induction on the planes. On the first plane $q_c \equiv 1$. On a later plane $M_c$ is
concave and non-decreasing (Lemma~\ref{lem:mean}) and is composed with concave functions, so
the composition is concave. Each cap in \eqref{eq:qproj} is affine in $(q_u, x_u)$ with a
positive coefficient on the concave $q_u$, hence concave, and so is their minimum. Monotonicity follows the same induction. At a binary $x$ the caps are
the wake inequalities and the sweep is that of Proposition~\ref{prop:field}(a).
\end{proof}

Antitonicity says that the field can only drop when turbines are added;
Proposition~\ref{prop:concave} adds that it drops concavely, in the cube of the speed. A
supergradient $g$ of $q_c$ at $x$ costs one reverse sweep: the forward sweep records which
term of \eqref{eq:qproj} is active at each cell, and the chain rule walks the planes backwards
from $c$ through the active terms.

\subsection{The Benders cut}\label{sec:model:cut}

At an integer master solution $\bar x$ with support $X$, take a direction and an installed
cell $c$ with $\bar z_c > \kappa V^3 q_c(\bar x)$, and let $g$ be a supergradient of $q_c$ at
$\bar x$. The \emph{Benders cut} is
\begin{equation}\label{eq:tangent}
z_c \le \kappa V^3 \bigl[\, q_c(\bar x) + g^\top (x - \bar x) \,\bigr],
\end{equation}
valid on $[0,1]^{\mathcal{C}}$ by Proposition~\ref{prop:concave} and $P(w) \le \kappa w^3$,
also when $q_c$ is the shifted field $\tilde q_c \ge q_c$ of Section~\ref{sec:model:cube},
on which the implementation separates it, and exact at $\bar x$ for the power
$\kappa V^3 \tilde q_c$ of the cubic program. For a fixed $x$ the field is the subproblem, and
Proposition~\ref{prop:concave} makes $\kappa V^3 q_c$ its concave value. Then
\eqref{eq:tangent} is the optimality cut of generalized Benders decomposition
\citep{geoffrion1972} on that value, one per cell and direction, with the supergradient of
the reverse sweep in the role of the dual multipliers. Both \eqref{eq:tangent} and the upwind cut \eqref{eq:upwind} involve
only the layout variables of the turbines that influence $c$, but they release the bound differently: in
\eqref{eq:upwind} one missing turbine of $U$ releases it entirely, in \eqref{eq:tangent} each
turbine carries its own marginal effect. The support of $g$ also gives a valid $U$ of
\eqref{eq:upwind} for free, usually a small one.

\begin{lemma}\label{lem:support}
Let $U = \{u \in X \setminus \{c\} : g_u \ne 0\}$ be the turbines of $X$ on which the
supergradient $g$ of \eqref{eq:tangent} is nonzero. Then $q_c$ takes the same value on
$U \cup \{c\}$ as on $X$, so the upwind cut \eqref{eq:upwind} with this $U$, whose $p$ is the
power of $c$ with turbines on $U \cup \{c\}$ only, is exact at $\bar x$.
\end{lemma}

\begin{proof}
Removing the turbines with $g_u = 0$ gives a binary $x'$ with $g^\top(x' - \bar x) = 0$,
hence $q_c(x') \le q_c(\bar x)$ by concavity and $q_c(x') \ge q_c(\bar x)$ by antitonicity.
\end{proof}

For a single upwind turbine $u$ the upwind cut with $U = \{u\}$ is the chord of the concave
function $x_u \mapsto \kappa V^3 q_c$ between its two integer points (the segment joining its values at $x_u = 0$ and $x_u = 1$), which is tighter than
any tangent of that function. The master therefore also carries the \emph{single-source chords}
$z_c \le \kappa V^3 - [\kappa V^3 - \kappa (V - d_{uc})^3]\, x_u$, the chord of the cubic
power without cut-in, valid since $P(v) \le \kappa v^3$ at every speed. There is one for
every ordered pair whose single-source deficit $d_{uc}$ (the speed lost at $c$ when $u$ alone
is installed) exceeds $10^{-2}$ m/s; they enter as lazy inequalities, computed on the
unshifted field. Against the upwind cut with the
$U$ of Lemma~\ref{lem:support}, on the same harness, the Benders cut is at parity in
time and takes fewer nodes (Section~\ref{sec:exp:attr}). We therefore present it as the
cleaner equivalent of the upwind cut, which is the logic-based Benders cut
\citep{hooker2003} of the same scheme, and not as a speed-up. Its role here is to make the master of
Section~\ref{sec:tiles} exact, whatever columns are added to it.

\section{A Dantzig--Wolfe reformulation on tiles}\label{sec:tiles}

The master of Section~\ref{sec:model} credits every installed cell with $P(V)$ and learns
the wakes one integer point at a time. We now tighten it with a convexification
\emph{computed} on small tiles of the grid: the power of every cell of a tile under every
layout of the tile is tabulated once, and the table enters the master as columns.

\subsection{Tiles, patterns and tables}\label{sec:tiles:tab}

A \emph{tile} $R \subseteq \mathcal{C}$ is a set of cells. We use \emph{strips}: the $b$
consecutive columns (vertical strips) or rows (horizontal strips) of the grid starting at the
offsets $0, k, 2k, \dots$ and at the last offset $n - b$, so that every cell is covered.
Consecutive strips of the same orientation overlap on $b - k$ lines, the last one on more
when $n - b$ is not a multiple of $k$. Vertical strips hold the along-wind interaction of the north and south winds,
horizontal strips that of the east and west winds, and only the lateral spread of a wake
crosses a border. The design of the paper has $b = 4$ and $k = 2$ (written V4s2+H4s2 in what follows): on the
even grids every cell lies in at most two vertical and two horizontal strips, namely,
\val{tiles14_v4s2h4s2_regions} tiles on the $14\times14$ grid (Figure~\ref{fig:tiles}(a)).
On the $17\times17$ grid of Section~\ref{sec:limits}, where $n - b$ is odd, the last offset
puts one line of cells in three strips.
A \emph{pattern} of $R$ is a set $T \subseteq R$ that satisfies the
separation and the budget of \eqref{eq:layout} inside the tile, the empty set and the
singletons included (Figure~\ref{fig:tiles}(b)). The patterns of a strip grow with its width much faster than with its
length---hence the width is the parameter that costs. The design above is the widest
whose patterns we enumerate on the whole ladder, \val{cor2_n14_full_patterns} of them at
$14\times14$ and \val{honest_n16_uniform_patterns} at $16\times16$.

\begin{figure}[t]
\centering
\begin{tikzpicture}[x=0.36cm,y=-0.36cm,font=\footnotesize,line cap=round]
  \begin{scope}
    \fill[gray!25] (4,0) rectangle (6,12);              %
    \foreach \i in {0,...,12} { \draw[gray!60,very thin] (0,\i) -- (12,\i); \draw[gray!60,very thin] (\i,0) -- (\i,12); }
    \draw[very thick] (2,0) rectangle (6,12);
    \draw[very thick,dashed] (4,0) rectangle (8,12);
    \draw[very thick,dotted] (0,4) rectangle (12,8);
    \node at (4,-0.6) {$R_1$};
    \node at (6.8,-0.6) {$R_2$};
    \node[right] at (12,6) {$R_3$};
    \draw[|-|] (2,-1.6) -- (4,-1.6) node[midway,above] {$k=2$};
    \draw[|-|] (2,12.7) -- (6,12.7) node[midway,below] {$b=4$};
    \node at (6,14.6) {(a)};
  \end{scope}
  \begin{scope}[shift={(16,0)}]
    \begin{scope}
      \clip (0,0) rectangle (12,12);
      \foreach \c/\r in {2/0,5/3,2/6,5/9} {
        \fill[gray!55] (\c,\r+1) rectangle (\c+1,\r+2);
        \fill[gray!35] (\c,\r+2) rectangle (\c+1,\r+4);
      }
      \fill[gray!18] (3,1) rectangle (4,3);
      \fill[gray!18] (4,4) rectangle (5,6);
      \fill[gray!18] (3,7) rectangle (4,9);
      \fill[gray!18] (4,10) rectangle (5,12);
    \end{scope}
    \foreach \i in {0,...,12} { \draw[gray!60,very thin] (0,\i) -- (12,\i); \draw[gray!60,very thin] (\i,0) -- (\i,12); }
    \draw[very thick] (2,0) rectangle (6,12);
    \foreach \c/\r in {2/0,5/3,2/6,5/9} {
      \fill (\c+0.5,\r+0.5) circle[radius=0.1cm];
    }
    \foreach \c/\r in {8/0,0/3,9/6,0/9,8/11} {
      \draw[fill=white] (\c+0.5,\r+0.5) circle[radius=0.1cm];
    }
    \node at (4,-0.6) {$R_1$};
    \draw[-{Stealth}] (9.5,-2.3) -- (9.5,-0.5) node[midway,right] {north wind};
    \node at (6,14.6) {(b)};
  \end{scope}
\end{tikzpicture}
\caption{The tiles. (a) Strips of width $b = 4$ at offsets $0, 2, 4, \dots$ on a
$12\times12$ grid: the vertical strips $R_1$ (solid) and $R_2$ (dashed) overlap on $b - k =
2$ columns (shaded), and the horizontal strip $R_3$ (dotted) crosses both. Every cell lies in
at most two vertical and two horizontal strips. (b) A pattern $T$ of $R_1$ (filled dots)
under the north wind. The power $p_c(T)$ of every cell $c$ of $R_1$ is tabulated once, with
the rest of the grid empty; the shading is schematic, darker cells losing more power.
Turbines outside the tile (hollow dots) can only lower those powers, which is what makes
the pattern powers valid bounds.}
\label{fig:tiles}
\end{figure}
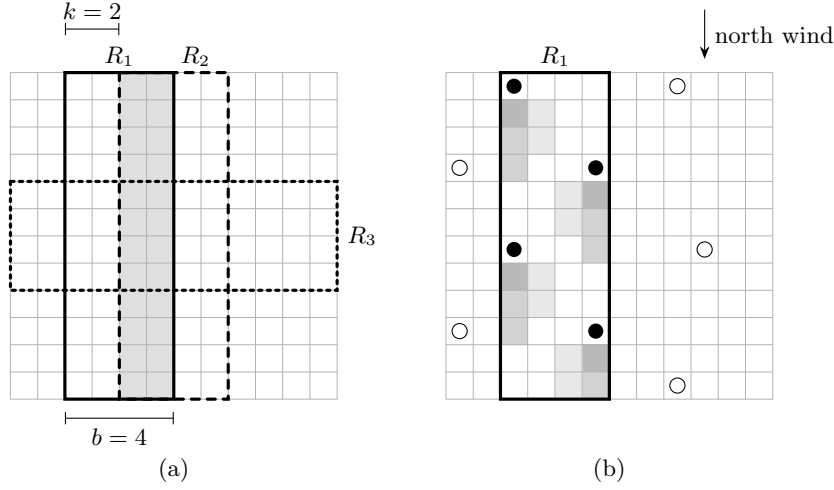

The \emph{table} of $R$ is the array $p^s_c(T)$, for $c \in R$, $s \in \mathcal{S}$ and $T$
a pattern of $R$: the power of cell $c$ in direction $s$ when exactly $T$ is installed
on the whole grid, in the cubic program with its shift. Note that
it is at least the exact power of the cell, and equal to it above cut-in up to the shift.
The table is computed by the sweep of Proposition~\ref{prop:field}(a) in the cube coordinate, run on
all the patterns of the tile at once, as one batched sweep in double precision on the CPU.
The sweep is restricted to the cells that a turbine of the tile reaches and that reach a
cell of the tile. The GPU serves the heuristic of Section~\ref{sec:heur} only.
The
tables of the design cost \val{cor2_n14_full_tables_s} s at $14\times14$ on the four-core
computers of Section~\ref{sec:exp:setup}. The tables are not translation invariant, since the free stream enters
at the border of the grid and the lateral spread of a wake feels it: every tile has its own.

\subsection{The columns}\label{sec:tiles:col}

For every tile $R$ we introduce a variable $\lambda^R_T \ge 0$ per pattern $T$ and add to the
master
\begin{equation}\label{eq:columns}
\sum_{T} \lambda^R_T = 1, \qquad
\sum_{T \ni c} \lambda^R_T = x_c \quad (c \in R), \qquad
z^s_c \le \sum_{T \ni c} p^s_c(T)\, \lambda^R_T \quad (c \in R,\ s \in \mathcal{S}).
\end{equation}
This is a Dantzig--Wolfe reformulation. The block of tile $R$ is the layout problem on the
tile alone: the pairs $(x_R, z_R)$ with $x_R$ the indicator of a pattern $T$,
$0 \le z^s_c \le p^s_c(T)$ for $c \in T$ and $z^s_c = 0$ for $c \in R \setminus T$. It is
the hypograph of the row $T$ of the table on the cells of $T$, closed downward, with no
credit outside $T$. System~\eqref{eq:columns} describes
the convex combinations of these hypographs, with the site variables $x$ as the linking
variables. The blocks overlap, as in the arrowhead form of \citet{bergner2015}, and $x$ ties
the tiles together cell by cell; it could be eliminated through the second equation, and we
keep it as the variable the solver branches on. What departs from the textbook scheme is the
block itself. Note that the block is a relaxation of the tile, since turbines outside $R$
lower the power of its cells. This is where antitonicity enters, and why a cut at integer
points is still needed.

\begin{proposition}\label{prop:columns}
(a) System \eqref{eq:columns} is valid: every feasible layout $X$ with its powers
$z^s_c = x_c\, p^s_c(X)$ extends to a solution of it. (b) At every binary $x$, $\lambda^R$ is the
indicator of the pattern $X \cap R$ and \eqref{eq:columns} reads $z^s_c \le p^s_c(X \cap R)$ for
$c \in X \cap R$ and $z^s_c \le 0$ for $c \in R \setminus X$.
(c) For any family of tiles, the master with the columns \eqref{eq:columns} and an
optimality cut exact at integer points, separated there, is an exact reformulation of the
program the cut is exact for: of the unshifted cubic program ($\varepsilon = 0$) with
the Benders cut \eqref{eq:tangent}, of
EndoWake with the upwind cut \eqref{eq:upwind} and the $U$ of Lemma~\ref{lem:support}. The
two programs coincide above cut-in.
\end{proposition}

\begin{proof}
(a) Take $\lambda^R$ the indicator of $X \cap R$, a pattern since $X$ is feasible; then
$\sum_{T \ni c} \lambda^R_T = x_c$, and the right-hand side of the power inequality is
$x_c\, p^s_c(X \cap R) \ge x_c\, p^s_c(X) = z^s_c$ by Proposition~\ref{prop:field}(c), as
$X \supseteq X \cap R$ (for $c \notin X$ both sides are zero).
(b) A cell with $x_c = 0$ excludes every pattern that contains it, a cell with $x_c = 1$ puts
the whole mass on the patterns that contain it, and the only pattern that survives both is
$X \cap R$; every term of the power inequality of a cell $c \notin X$ has weight zero. (c) By (b) the columns only add valid upper bounds at integer points, and the
cuts restore $z^s_c \le p^s_c(X)$ there: the upwind cut in EndoWake, the Benders cut in the
unshifted cubic program.
\end{proof}

The implementation runs the cubic program with its shift, and
adds the single-source chords of Section~\ref{sec:model:cut}, computed on the unshifted
field. Every inequality of its master is valid at every layout of the model, so the master
is a relaxation of the model, and not literally the shifted cubic program; the certificate of
Appendix~\ref{app:proofs} uses the former property only.

In short, the tiles tighten the relaxation and the cut keeps the master exact at integer
layouts, for the unshifted cubic program: the pattern power carries the wakes
inside the tile, and what crosses its border comes back through the Benders cuts. The same
$\lambda$ supports three families of power inequalities, each implied by the next: the
aggregate one, $\sum_{c \in R} z^s_c \le \sum_T P^s_R(T)\, \lambda^R_T$ with
$P^s_R(T) = \sum_{c \in T} p^s_c(T)$, for the direction $s$ along the strip; one per cell in
that direction; one per cell and direction, as in \eqref{eq:columns}. We use the last: it
costs nonzeros and no variables, and at the root it was the tightest of the three on the
designs we tried (strips of width 1 to 3 on the $10\times10$ grid, not tabulated).
The band inequalities of \citet{fischetti2026endowake} are the aggregate strength on a band
narrower than the separation, in its own direction, projected once more on the number of
turbines of the band. They are the chords of the upper hull of $(j, \Psi_R(j))$, with $\Psi_R(j)$ a
dynamic-programming bound on the power of $j$ turbines of the band. By contrast, projecting
\eqref{eq:columns} on $(x_R, \sum_{c \in R} z^s_c)$ gives the concave envelope of the
band power of the table, shifted as in Section~\ref{sec:model:cube}, in the space of the
cells. It is attained on every pattern, so among the inequalities that hold on the hypographs
of the shifted blocks none is tighter in the space of those variables. The same holds for the
unshifted cubic model, with the table at $\varepsilon = 0$ and no cut-in, since a pattern
alone is a feasible layout that attains its unshifted value.

\subsection{Loss coordinates: the sparse master}\label{sec:tiles:loss}

The pattern $\{c\}$ leaves the free stream at $c$, so $\max_{T \ni c} p^s_c(T) = P(V)$ and,
through $\sum_{T \ni c} \lambda^R_T = x_c$, the power inequality of \eqref{eq:columns} is the
same as
\begin{equation}\label{eq:loss}
z^s_c \le P(V)\, x_c - \sum_{T \ni c} \bigl[ P(V) - p^s_c(T) \bigr] \lambda^R_T ,
\end{equation}
the cap of the master minus the loss of each pattern at $c$. A nonzero appears only where a
pattern loses power at $c$. The LP is the same, with
\val{cor2_n14_loss_nz_reduction_pct}\,\% fewer nonzeros at $14\times14$; what the sparse
form saves in root time and memory is measured in Section~\ref{sec:exp:attr}.

\subsection{Which tiles}\label{sec:tiles:design}

Width and overlap both buy bound, and patterns pay for both. Table~\ref{tab:design}
compares three designs on the $14\times14$ grid with the pipeline of
Sections~\ref{sec:shrink}--\ref{sec:heur}, on five seeds. Every design receives the same
incumbent (the certified optimum, entered as it is) and runs on the shared nodes of a
cluster (two Xeon models, Section~\ref{sec:exp:setup}, not a computer of the ladder), so that the
designs differ in the root, the fixing and the tree only. The work units of the solver are
the comparison there, and the seconds are indicative. Strips of width 3 at every offset (V3+H3, the first
design that certified this grid) leave a root gap of \val{tiles14inc_v3h3_root_gap_pct}\,\%
with \val{tiles14inc_v3h3_patterns} patterns. The design of the paper leaves
\val{tiles14inc_v4s2h4s2_root_gap_pct}\,\% with \val{tiles14inc_v4s2h4s2_patterns}. Width 5
at every third offset leaves \val{tiles14inc_v5s3h5s3_root_gap_pct}\,\% for
\val{tiles14inc_v5s3h5s3_patterns_ratio} times the patterns of the design. The designs
change width and overlap together, so the table ranks them and does not separate the two.
Two facts fixed the choice. Overlap is what the node LPs pay: with width 4 at every offset the dual simplex
stalls on degenerate pivots that each price every column, and at overlap 2 the stalling goes
while the bound stays. A tight root makes the reduced-cost fixing of
Section~\ref{sec:shrink} extremely effective, since its threshold is the gap between the root and
the incumbent. At the same incumbent the tree of the design takes
\val{tiles14inc_v4s2h4s2_work} work units against \val{tiles14inc_v3h3_work} for width 3
(\val{tiles14inc_v4s2h4s2_tree_s} against \val{tiles14inc_v3h3_tree_s} s). At $16\times16$,
with the certified optimum as the incumbent, width 3 certifies \val{tiles16inc_v3h3_certified}
of \val{tiles16inc_v3h3_runs} seeds within 10\,800 s, on
\val{tiles16inc_v3h3_nodes_min}--\val{tiles16inc_v3h3_nodes_max} nodes from a root gap of
\val{tiles16inc_v3h3_root_gap_pct}\,\%. The design of the paper closes the five seeds
from a cold start (Table~\ref{tab:16}). Width 5 buys a
few hundredths of a percent of root gap and the same tree
(\val{tiles14inc_v5s3h5s3_work} work units), and pays for them with a root LP of
\val{tiles14inc_v5s3h5s3_root_s} s and \val{tiles14inc_v5s3h5s3_maxrss_gb} GB of memory. In view of the above, the design is fixed from here on, and Section~\ref{sec:exp} measures the
pipeline with it. We decided to run the same tile width and stride on every grid and rose of
Sections~\ref{sec:exp} and~\ref{sec:limits}, with no selection per instance.

\begin{table}[t]
\centering\small
\caption{Three tile designs on the $14\times14$ grid, five seeds each, the certified
optimum as the incumbent of every design, shared Xeon nodes of two models, tree limit
1800 s: certified runs, root gap of the design (\%), columns after the reduced-cost fixing
and after the propagation between tiles, nodes, work units of the tree, tree and wall
seconds (medians; the work units are deterministic for a fixed model on one machine at equal settings, the seconds indicative), peak memory (GB).}
\label{tab:design}
\footnotesize\setlength{\tabcolsep}{4pt}
\begin{tabular}{l r r rr rrrr r}
\toprule
design & cert. & root gap & fixing & propagation & nodes & work & tree s & wall s & GB\\
\midrule
V3+H3 & \val{tiles14inc_v3h3_certified}/\val{tiles14inc_v3h3_runs} & \val{tiles14inc_v3h3_root_gap_pct} & \val{tiles14inc_v3h3_kept} & \val{tiles14inc_v3h3_kept_final} & \val{tiles14inc_v3h3_nodes} & \val{tiles14inc_v3h3_work} & \val{tiles14inc_v3h3_tree_s} & \val{tiles14inc_v3h3_wall_s} & \val{tiles14inc_v3h3_maxrss_gb}\\
V4s2+H4s2 & \val{tiles14inc_v4s2h4s2_certified}/\val{tiles14inc_v4s2h4s2_runs} & \val{tiles14inc_v4s2h4s2_root_gap_pct} & \val{tiles14inc_v4s2h4s2_kept} & \val{tiles14inc_v4s2h4s2_kept_final} & \val{tiles14inc_v4s2h4s2_nodes} & \val{tiles14inc_v4s2h4s2_work} & \val{tiles14inc_v4s2h4s2_tree_s} & \val{tiles14inc_v4s2h4s2_wall_s} & \val{tiles14inc_v4s2h4s2_maxrss_gb}\\
V5s3+H5s3 & \val{tiles14inc_v5s3h5s3_certified}/\val{tiles14inc_v5s3h5s3_runs} & \val{tiles14inc_v5s3h5s3_root_gap_pct} & \val{tiles14inc_v5s3h5s3_kept} & \val{tiles14inc_v5s3h5s3_kept_final} & \val{tiles14inc_v5s3h5s3_nodes} & \val{tiles14inc_v5s3h5s3_work} & \val{tiles14inc_v5s3h5s3_tree_s} & \val{tiles14inc_v5s3h5s3_wall_s} & \val{tiles14inc_v5s3h5s3_maxrss_gb}\\
\bottomrule
\end{tabular}
\end{table}

\section{Shrinking the master before the tree search}\label{sec:shrink}

The columns are many and the tree must not branch on them. Everything in this section rests
on one argument, which we state once here and prove with its hypotheses in
Appendix~\ref{app:proofs}. Let $v$ be the value of the root LP of the master and
$\bar c_T \le 0$ the reduced cost of column $\lambda^R_T$ at its optimal basis. A feasible
layout $X$ with its exact powers and the indicators of its patterns is a feasible point of
the root LP with objective $E(X)$, so LP duality gives $E(X) \le v + \bar c_T$ for every
pattern $T$ that $X$ uses. Hence, given an incumbent of value $L$, every column with
$\bar c_T < -(v - L)$ can be removed: no layout better than the incumbent uses it. This is the
reduced-cost fixing of column generation; the structure of the tiles adds a second test.

\paragraph{The incumbent.} It comes from the tile heuristic of Section~\ref{sec:heur},
run before the master is built. Where no GPU is at hand (the computers of the ladder and of
the roses), an LP dive on the columns of best reduced cost fixes the fractional $x_c$ of
largest value to one. The dive re-solves the LP by barrier (undoing a fix that makes it
infeasible) until the layout is integral or its time limit ends it. It then completes the
layout greedily, by the remaining LP values, and polishes it by 1-opt moves on the exact
energy. The
dive is at
the optimum up to the $14\times14$ grid and at \val{cor2_n16_full_inc} against the optimum
\val{cor2_n16_full_exact} on the $16\times16$. Either way the incumbent enters the solver as
a complete MIP start, $x$ and $z$ at its exact powers (without the $z$ the lazy Benders cuts
reject it); its exact energy is the value $L$ of the fixing, and the solver prunes against its
own incumbent.

\paragraph{Fixing at the root.} The root LP is solved once, by the solver's default method
with the settings of the run (the seed of the run acts in the tree). The reduced costs are those of the optimal basis it returns (another optimal basis of
the same degenerate LP gives another kept set, Section~\ref{sec:exp:attr}). The columns with $\bar c_T < -(v - L)$, up to a tolerance of \val{tol_fix} of the root value that protects the patterns
of an optimal incumbent from rounding, leave the tables. From the dive, the fixing keeps
\val{cor2_n14_loss_kept} of \val{cor2_n14_loss_patterns} columns at $14\times14$; with the
heuristic at the optimum it keeps \val{honest_n16_uniform_kept} of the
columns of the $16\times16$ grid.

\paragraph{Propagation between tiles.} Two overlapping tiles must agree on their shared
cells: a surviving pattern of tile $R$ that agrees on $R \cap R'$ with no surviving pattern
of a neighbor $R'$ cannot be part of any layout of the reduced master, and can be removed too.
Each tile is a table constraint on its cells in the sense of constraint programming, whose
allowed tuples are its surviving patterns, and the test is pairwise consistency between
the constraints of overlapping tiles \citep{bessiere2006}, with the whole intersection as the
constraint; \citet{lecoutre2013} enforce it on table constraints through the assignments of
the intersections. In the terms of Dantzig--Wolfe it is the support form of the consistency cuts of
\citet{clausen2022}: for two tiles that share $S = R \cap R'$, their equations
$\sum_{T:\, T \cap S = A} \lambda^R_T = \sum_{T':\, T' \cap S = A} \lambda^{R'}_{T'}$, one per
assignment $A$ of the shared cells, equate the mass of the two tiles on every pattern of the
overlap. An assignment with no surviving pattern on one side sets the patterns of the other
side to zero, which is what the propagation does; at fractional points the equations also
equate the mass that the two tiles put on every assignment of the overlap, whereas the
master equates only the marginal of each shared cell, through the site variables $x$. A two-cell form of them, equal
mass on the patterns that hold two shared cells, is an option of the code that the pipeline
does not use. The signatures of the shared cells are built per pair of overlapping tiles, and
the passes repeat to a fixed point (the implementation caps their time), in a matter of
seconds. It is valid as the fixing is: a layout whose
patterns are all present stays present. On the ladder it is the lever that shortens the tree
(Section~\ref{sec:exp:attr}); at $16\times16$ it leaves \val{honest_n16_uniform_kept_final}
of the survivors of the fixing.

\paragraph{Refixing in the tree.} The value $L$ is the best exact energy known so far. Every new
incumbent is re-evaluated by the oracle, and its energy for the pipeline is the smaller of
its value in the shifted sweep and its exact energy, in practice the latter. When it beats
the incumbent value $L$, the test is re-applied with
the root reduced costs and the new threshold. The columns it removes get $\lambda^R_T \le 0$ as lazy inequalities, since
bounds cannot change in a callback. The inequalities are valid as the fixing is. Note that it acts only on a weak
incumbent: it fires zero times on the ladder, where the dive is at the optimum, and in the
single job of Section~\ref{sec:exp}, where the heuristic is. On the $16\times16$ grid
started from the dive the first improving solution removes \val{cor2_n16_full_refixed} of
the \val{cor2_n16_full_kept} surviving columns in one callback.

\paragraph{The pipeline in order.} (1) The tile heuristic on the GPU, where there is one,
its layout evaluated by the oracle. (2) The tables of every tile. (3) The root LP of the
master with the columns, in loss coordinates. (4) Where there is no GPU, the LP dive on that
root, its layout evaluated by the oracle. (5) The reduced-cost fixing against the incumbent.
(6) The propagation between tiles, to a fixed point. (7) The tree on the reduced master,
with the Benders cuts as lazy inequalities at integer points, every improving incumbent
re-evaluated by the oracle, and the refixing. (8) The certificate, when the tree closes.

\paragraph{What is a certificate.} A tree that closes on the reduced master with the
incumbent entered at its exact energy. Every reduction above keeps the patterns of every
layout better than the incumbent, so a bound on the reduced master is a bound on every
layout better than the incumbent. The optimum of the model then lies between the incumbent
and that bound (Proposition~\ref{prop:valid} in Appendix~\ref{app:proofs}). A run that ends below the
incumbent of the fixing is not a certificate.

\section{A tile heuristic on the GPU}\label{sec:heur}

We next describe the heuristic that produces the incumbent. GPU batch evaluation of EndoWake
layouts, as the screen of a local search in a design pipeline, is the subject of
\citet{fischetti2026gpu}; here the GPU evaluates complete tile replacements and its output is
the incumbent of the reformulation. The heuristic is a local search whose move
is the tile itself: take a tile, keep the
turbines outside it, and choose the turbines inside it by enumeration.
Every pattern of the tile at least the separation distance away from the outside turbines
and within the budget is a candidate layout. All of them are scored at once by one
batched sweep of the wind field on the GPU, namely, the arithmetic of
Proposition~\ref{prop:field}(a) run on a batch of layouts in single precision. The score is
the unshifted cubic power of the field, capped at rated and without cut-in. For every direction the cells are
sorted by levels of the dependency graph, and a level costs one gather, one power and one scatter-minimum for the whole batch.
The compatible patterns are filtered on the GPU as a table of bits, one row per pattern. The
best \val{heur_topk_w} candidates of a visit are re-evaluated by the reference evaluator on the CPU, with
the power curve of Section~\ref{sec:model:inst} in double precision, so that the accepted
value is always exact. What the score of the ranking can cost is a candidate left out of
those \val{heur_topk_w}, never a wrong value. A candidate more than \val{heur_gputol} kW below the best exact value in the
ranking is skipped. A move changes up to $\Nmax$ turbines at once, where a 1-opt move changes one. The
tiles, the strips of width \val{heur_width} at every offset plus
\val{heur_win_rows}$\times$\val{heur_win_cols} windows at stride \val{heur_stride}, are visited
in random order until a round moves nothing, or for \val{heur_rounds_w} rounds at most. Then a perturbation removes from the best
layout the \val{heur_kick_w} turbines nearest to a random cell, the rounds start again, and
the best layout is kept. A start stops after \val{heur_patience_w} perturbations without a
better one, or \val{heur_kicks} in all, and the whole run stops at \val{heur_tl} s.
The search runs from \val{heur_starts_w} random greedy starts per run, since it turns out that a start
that fails ends a few tenths of a percent below the optimum.

In the single job of Section~\ref{sec:exp} the heuristic reaches the optimum of the
$16\times16$ grid on every seed, from \val{honest_n16_uniform_starts_at_exact} of
\val{honest_n16_uniform_starts} random starts. It takes a median of
\val{honest_n16_uniform_heur_s} s for \val{heur_starts_w} starts on one GPU. The same code
also runs without the GPU, on four cores of a Xeon Platinum 8592+ node, with a limit of
\val{heurcpu_cpu_tl} s. On all \val{heurcpu_same} seeds of the uniform rose the two runs
follow one trajectory, start by start and move by move, as far as the CPU run gets within its
limit, and the time to the last common move is \val{heurcpu_ratio_min}--\val{heurcpu_ratio_max}
times longer on the CPU. It
reaches the optimum in \val{heurcpu_cpu_t_opt_min}--\val{heurcpu_cpu_t_opt_max} s, against
\val{heurcpu_gpu_t_opt_min}--\val{heurcpu_gpu_t_opt_max} s on the GPU, and within the limit of
the job on \val{heurcpu_at_gpu_tl_opt} seeds of \val{heurcpu_seeds}. In these runs the GPU
saves time and does not change the layout. On the two
non-uniform roses of that section, it reaches the certified optimum from
\val{honest_n16_p60_starts_at_exact} of \val{honest_n16_p60_starts} and from
\val{honest_n16_p40_starts_at_exact} of \val{honest_n16_p40_starts} starts, again on every seed.
When every start of a run misses the optimum, the pipeline runs from a
weaker incumbent, as the cold start of Section~\ref{sec:exp:16} does on the uniform rose,
where it still certifies every seed. The parameters of the heuristic are one set, written in
the script of the single job before its first run and used unchanged in all its runs, on
every grid and rose.

\section{Computational experiments}\label{sec:exp}

We measure three things. The first is the EndoWake ladder against
its branch-and-check re-run on the same hardware (Section~\ref{sec:exp:ladder}).
The second is the $16\times16$ grid from random starts in one job, on the uniform rose and on two
non-uniform ones (Section~\ref{sec:exp:16}). The third is what each lever contributes
(Section~\ref{sec:exp:attr}).

\subsection{Setup}\label{sec:exp:setup}

The ladder, the roses and the attribution at $14\times14$ run on the computers of \citet{fischetti2026endowake} (Intel Xeon
E3-1220 v2 at 3.10 GHz, four cores without hyper-threading, 14 GB, one job per computer, four
threads). The two codes are thus compared at equal CPU and threads. The $16\times16$ grid
needs more memory and runs on the nodes of a cluster. The cold start runs on an
Intel Xeon Platinum 8592+ node, and the single job on AMD EPYC 9335 nodes with one NVIDIA RTX PRO 6000
Blackwell GPU (96 GB), four CPUs allocated to the job. Where the CPUs differ the comparison
within the pipeline is in the work units of the solver, a measure of its effort that is
deterministic for a fixed model at equal settings, thread count and hardware, as the solver
documents it (the fixing, hence the model of the tree, follows the basis the root LP returns,
Section~\ref{sec:shrink}), and indicative across machines; the wall times are stated per machine. The branch-and-check at $16\times16$ reports wall
time, gap and nodes only, since the frozen code records no work units. The design
comparison of Table~\ref{tab:design} runs on the shared nodes of the same cluster, Intel
Xeon Platinum 8592+ and Xeon Gold 6252N, where again the work units are the comparison.
Both codes use the same version of the solver, gurobipy 13.0.3. The branch-and-check is the frozen
code of \citet{fischetti2026endowake}, run as published (no MIP start and no cutoff, the
band inequalities in its master, the same seeds as the pipeline). At $16\times16$ it also
runs on the node type of the cold start, as published and with the optimum as cutoff. The pipeline is the design
V4s2+H4s2 in loss coordinates, reduced-cost fixing, propagation between tiles and refixing,
with the tile heuristic of Section~\ref{sec:heur} on the GPU where there is one and the LP
dive on the computers of the ladder. Every grid runs on five seeds, 1 to 5 (three, 1 to 3, for the
$14\times14$ roses). The seed is the seed parameter of the solver in
both codes and the generator of the heuristic. The limit is 3600 s on the ladder for both codes. For the pipeline it is the limit of the
tree, after the tables, the root and the incumbent, whose times the wall time includes. For
the branch-and-check it is the limit of its tree, after its master is built and a first
relaxation solved. At $16\times16$ the limit is 10\,800 s, on the tree of each code alike. The single job ends within
\val{honest_n16_p60_wall_max_s} s of wall time at most, heuristic and tables included.

A run is \emph{certified} when the solver closes the reduced master, with the incumbent
entered at its exact energy, within its default relative gap of $10^{-4}$. A closed tree
proves that no layout has energy above the final bound $U \ge L$ of the solver, so that the
optimum lies in $[L, U]$ (Appendix~\ref{app:proofs}). Throughout, the \emph{optimum} of a
grid is the energy $L$ of its certified layout, the best value within the gap. Three
quantities enter the \emph{width} $U - L$ of a certificate: the exact energy $L$ of the
final layout, the value $M$ that the master credits to it, and the bound $U$, with
$U - L = (U - M) + (M - L)$. The first term is the solver's gap, at most $10^{-4}$ of $M$
by the setting. The second is the \emph{residual}, the excess of the value credited over the
exact energy. When the solver's incumbent is a solution it found, possibly the same layout as
the start, that solution is credited about its shifted powers (Section~\ref{sec:model:cube}),
so $M$ exceeds $L$ by about the shift. When the solver keeps the start solution, entered at
its exact energy, the credit is $L$ and the residual is zero. Over the \val{shift_certified}
certified runs of the campaigns of this paper whose final layout is stored, the solver keeps
the start in \val{shift_start_runs} and ends on a solution of its own in
\val{shift_tree_runs}. Over the latter the residual is
\val{shift_resid_min_pct}--\val{shift_resid_max_pct}\,\%, at most
\val{shift_resid_max_pct_up}\,\%, and the credited value is within
\val{shift_obj_minus_shift_max_kw} kW of the sweep of the field at the shift $\varepsilon$. The same sweep
at $\varepsilon = 0$ agrees with the oracle to \val{shift_unshifted_vs_ref_max_kw} kW, and
on every stored layout the shift raises the energy by
\val{shift_shift_min_pct}--\val{shift_shift_max_pct}\,\%. The width $(U - L)/L$ is at most \val{cor2_gap_exact_max_pct}\,\% over the
\val{cor2_wc_certified} certified runs of \val{cor2_wc_runs} of the campaign of
Section~\ref{sec:exp:attr} (three arms on three grids of the ladder and three arms at
$16\times16$, five seeds each). Where the tree closes at the gap rather than at equality,
the width exceeds the residual by the solver's gap. An \emph{arm} is one configuration of
the pipeline: \texttt{base}, \texttt{loss} and \texttt{full} add the levers one at a time,
and Section~\ref{sec:exp:attr} defines them. For the single job of
Section~\ref{sec:exp:16} the width is \val{honest_n16_uniform_gap_exact_pct}\,\%. Repeated
with the gap set to zero, the job closes with the same width, the same nodes and the same
work units: there $U = M$, the gap does not govern it, and the width is the residual alone.
The threshold of the separation, which adds a cut only when a power variable exceeds its
power in the cubic program by more than \val{tol_sep} of the rated power, is a smaller source on
top. These are the certificates of a floating-point solver: we compute no safe bound in the
sense of \citet{neumaier2004}, and the width is the margin we report. A run at the limit is
censored and enters no ratio. The spread within a grid and an arm is across seeds.

\subsection{The ladder}\label{sec:exp:ladder}

Table~\ref{tab:ladder} gives the ladder. The branch-and-check certifies the $10\times10$ grid
in \val{cor2_n10_ejor_solve_s} s and none of the larger ones in an hour. At the limit it
stands at \val{cor2_n12_ejor_gap_pct}\,\% and \val{cor2_n14_ejor_gap_pct}\,\% of gap. It finds the
$12\times12$ optimum on \val{cor2_n12_ejor_found_opt} of \val{cor2_n12_ejor_runs} seeds and
stays \val{cor2_n14_ejor_best_below_pct}\,\% under the $14\times14$ one. Its gap is a bound
effect before an incumbent effect: the final bound of its published run
\citep{fischetti2026endowake} stays \val{bound_n12_ejorpub_bound_above_pct}\,\% and
\val{bound_n14_ejorpub_bound_above_pct}\,\% above the optima certified here, so that the
optimum in place of its incumbent, at that bound, leaves a gap of that size. The cutoff run
at $16\times16$ (Table~\ref{tab:16}) measures the proof effort that remains once the optimum
is known; what a warm start would do to its search we do not infer. The pipeline
certifies every seed of every grid. The $10\times10$ and the $12\times12$ close at the root,
the $14\times14$ in a median of \val{cor2_n14_full_wall_s} s of wall time, all phases
included, with a tree of \val{cor2_n14_full_nodes} nodes on \val{cor2_n14_full_kept} columns.
The root gap of the design is below half a percent on every grid, and the fixing and
propagation leave a few percent of the columns to the tree.

\begin{table}[t]
\centering\small
\caption{The ladder on the computers of Section~\ref{sec:exp:setup}, five seeds per grid, 3600 s on the tree of each code; pipeline with the
\texttt{full} arm of Table~\ref{tab:attr} and the LP dive as incumbent. Branch-and-check:
certified runs, median solve time or median gap at the limit, best energy found below the
optimum (\%). Pipeline: certified runs, root gap of the design (\%), columns in the tree after
fixing and propagation (median), nodes, tree and wall seconds (medians), peak memory (GB).}
\label{tab:ladder}
\begin{tabular}{l rrr rrrrrrr}
\toprule
& \multicolumn{3}{c}{branch-and-check} & \multicolumn{7}{c}{pipeline}\\
\cmidrule(lr){2-4}\cmidrule(lr){5-11}
grid & cert. & s or gap & below & cert. & root gap & columns & nodes & tree s & wall s & GB\\
\midrule
$10\times10$ & \val{cor2_n10_ejor_certified}/\val{cor2_n10_ejor_runs} & \val{cor2_n10_ejor_solve_s} s & \val{cor2_n10_ejor_best_below_pct}
 & \val{cor2_n10_full_certified}/\val{cor2_n10_full_runs} & \val{cor2_n10_full_root_gap_pct} & \val{cor2_n10_full_kept} & \val{cor2_n10_full_nodes} & \val{cor2_n10_full_tree_s} & \val{cor2_n10_full_wall_s} & \val{cor2_n10_full_maxrss_gb}\\
$12\times12$ & \val{cor2_n12_ejor_certified}/\val{cor2_n12_ejor_runs} & \val{cor2_n12_ejor_gap_pct}\,\% & \val{cor2_n12_ejor_best_below_pct}
 & \val{cor2_n12_full_certified}/\val{cor2_n12_full_runs} & \val{cor2_n12_full_root_gap_pct} & \val{cor2_n12_full_kept} & \val{cor2_n12_full_nodes} & \val{cor2_n12_full_tree_s} & \val{cor2_n12_full_wall_s} & \val{cor2_n12_full_maxrss_gb}\\
$14\times14$ & \val{cor2_n14_ejor_certified}/\val{cor2_n14_ejor_runs} & \val{cor2_n14_ejor_gap_pct}\,\% & \val{cor2_n14_ejor_best_below_pct}
 & \val{cor2_n14_full_certified}/\val{cor2_n14_full_runs} & \val{cor2_n14_full_root_gap_pct} & \val{cor2_n14_full_kept} & \val{cor2_n14_full_nodes} & \val{cor2_n14_full_tree_s} & \val{cor2_n14_full_wall_s} & \val{cor2_n14_full_maxrss_gb}\\
\bottomrule
\end{tabular}
\end{table}

The same ladder under six non-uniform axial roses (opposite pairs equal, one prevailing
direction, all different, mild and strong) runs on the same computers, with the LP dive as
incumbent and the \texttt{base} arm of Table~\ref{tab:attr}. That arm is the fixing at the
root before the loss coordinates and the propagation were in place, with no symmetry device.
The six roses are, as $(p_N, p_E, p_S, p_W)$, \val{rose_nu_list}. Overall, \val{rose_solved} runs of
\val{rose_runs} close, and at $12\times12$ and $14\times14$ the uniform rose is the hardest
(Table~\ref{tab:rose}). With the cardinality cap tight the optimum is the uniform optimum
reoriented in \val{rose_d4_images} cases of \val{rose_rows}. The exception is the
$10\times10$ grid at 90\,\% on one axis, where the optimum installs
\val{rose_n10_axial_turbines} turbines against \val{rose_n10_uniform_turbines}.

\begin{table}[t]
\centering\small
\caption{The ladder on the six non-uniform axial roses listed in the text, same computers, LP
dive as incumbent, \texttt{base} arm of Table~\ref{tab:attr}, five seeds per rose at $10\times10$ and $12\times12$ and three at
$14\times14$ (two roses of the $14\times14$ solve the root by barrier: their times are
indicative). Per grid: nodes and tree seconds of the uniform rose under the same arm and on the same
seeds (three at $14\times14$, hence medians other than those of Table~\ref{tab:attr}), and over
the six roses the runs closed, the range of the median nodes, tree seconds and root gap (\%),
and the roses whose optimum is a dihedral image of the uniform one.}
\label{tab:rose}
\begin{tabular}{l rr r rrrr}
\toprule
& \multicolumn{2}{c}{uniform} & \multicolumn{5}{c}{six roses}\\
\cmidrule(lr){2-3}\cmidrule(lr){4-8}
grid & nodes & tree s & closed & nodes & tree s & root gap & images\\
\midrule
$10\times10$ & \val{rose_n10_uniform_nodes} & \val{rose_n10_uniform_tree_s} & \val{rose_n10_nu_solved}/\val{rose_n10_nu_runs} & \val{rose_n10_nu_nodes_min}--\val{rose_n10_nu_nodes_max} & \val{rose_n10_nu_tree_min_s}--\val{rose_n10_nu_tree_max_s} & \val{rose_n10_nu_root_gap_min_pct}--\val{rose_n10_nu_root_gap_max_pct} & \val{rose_n10_nu_d4_images}/\val{rose_n10_nu_roses}\\
$12\times12$ & \val{rose_n12_uniform_nodes} & \val{rose_n12_uniform_tree_s} & \val{rose_n12_nu_solved}/\val{rose_n12_nu_runs} & \val{rose_n12_nu_nodes_min}--\val{rose_n12_nu_nodes_max} & \val{rose_n12_nu_tree_min_s}--\val{rose_n12_nu_tree_max_s} & \val{rose_n12_nu_root_gap_min_pct}--\val{rose_n12_nu_root_gap_max_pct} & \val{rose_n12_nu_d4_images}/\val{rose_n12_nu_roses}\\
$14\times14$ & \val{rose_n14_uniform_nodes} & \val{rose_n14_uniform_tree_s} & \val{rose_n14_nu_solved}/\val{rose_n14_nu_runs} & \val{rose_n14_nu_nodes_min}--\val{rose_n14_nu_nodes_max} & \val{rose_n14_nu_tree_min_s}--\val{rose_n14_nu_tree_max_s} & \val{rose_n14_nu_root_gap_min_pct}--\val{rose_n14_nu_root_gap_max_pct} & \val{rose_n14_nu_d4_images}/\val{rose_n14_nu_roses}\\
\bottomrule
\end{tabular}
\end{table}

\subsection{The $16\times16$ grid from random starts}\label{sec:exp:16}

Table~\ref{tab:16} collects the $16\times16$ runs, with the source of the incumbent
declared and the seeds counted in its first column. The branch-and-check, re-run on the
computers of the ladder, closes no seed in three
hours and stays at \val{ejor16_tl10800_gap_pct}\,\% of gap with its best layout
\val{ejor16_tl10800_best_below_pct}\,\% under the optimum. On the node type of the cold start
it closes no seed either, as published or with the optimum as cutoff (raised by the relative
gap of $10^{-4}$), which leaves it only the proof. On every seed its bound stays at
\val{ejor16up_cutoff_bound} kW, \val{ejor16up_cutoff_gap_pct}\,\% above the optimum, as on
the computers of the ladder. The pipeline from a cold start,
with the LP dive as incumbent (\val{honest_ref_cor2full_inc_gap_pct}\,\% under the optimum)
and no GPU, certifies the five seeds within the limit; the refixing fires at the first
improving solution and removes \val{cor2_n16_full_refixed} columns. The single job is the
heuristic on the GPU followed by the same arm under one timer, with no stored layout, no
cached root and no known optimum. It certifies the five seeds at \val{honest_n16_uniform_exact}
kW in \val{honest_n16_uniform_wall_min_s}--\val{honest_n16_uniform_wall_max_s} s of wall
time and \val{honest_n16_uniform_cpu_h} CPU hours, with the GPU allocated for the whole job
and busy for the heuristic only. Repeated with the solver's gap at zero on the five seeds, it closes
them again with the same width (\val{honest_n16_uniform_gap0_gap_exact_pct}\,\%), the same
nodes and the same work units. On the rose $(p_N, p_E, p_S, p_W) = (0.6, 0.25, 0.1, 0.05)$,
whose optimum was unknown before, the single job certifies the
five seeds at \val{honest_n16_p60_exact} kW in
\val{honest_n16_p60_wall_min_s}--\val{honest_n16_p60_wall_max_s} s: the root gap is larger
(\val{honest_n16_p60_root_gap_pct} against \val{honest_n16_uniform_root_gap_pct}\,\%) and,
as expected, the tree pays for it. The rose $(0.4, 0.3, 0.2, 0.1)$ sits between the two:
the five seeds certify at \val{honest_n16_p40_exact} kW in
\val{honest_n16_p40_wall_min_s}--\val{honest_n16_p40_wall_max_s} s, with a root gap of
\val{honest_n16_p40_root_gap_pct}\,\% and \val{honest_n16_p40_nodes} nodes, against
\val{honest_n16_uniform_nodes} on the uniform rose and \val{honest_n16_p60_nodes} on the
skewed one. Between the cold start and the single job the root costs the same work
(\val{honest_ref_cor2full_root_work} against \val{honest_n16_uniform_root_work} units) and
the tree \val{honest_ref_cor2full_work_ratio} times less, on \val{honest_ref_cor2full_kept_ratio}
times fewer columns. This is what the incumbent at the optimum buys.

\begin{table}[t]
\centering\small
\caption{The $16\times16$ grid, five seeds per row; the first column counts the certified
runs. Branch-and-check on the computers of the ladder at two limits on its tree (wall time only) and for
10\,800 s on the Xeon node, as published and with the optimum as cutoff (its gap: its bound
against the exact energy of its own best layout); the pipeline from a cold start (dive
incumbent, Xeon node, tree limit 10\,800 s) and in the single job (heuristic on the GPU, then
the same arm, EPYC node), also with the solver's gap at zero and on two non-uniform roses.
Incumbent: its gap to the optimum (\%) and, for the heuristic, the random starts that reach
the optimum; columns in the tree (median); nodes, tree
seconds, work units and wall seconds (medians, wall as a range; the median wall of the single
job is \val{honest_n16_uniform_wall_s} s on the uniform rose and \val{honest_n16_p60_wall_s} s
on the rose 0.6/0.25/0.1/0.05); peak memory (GB).}
\label{tab:16}
\footnotesize\setlength{\tabcolsep}{3pt}
\begin{tabular}{l l r r r r r r r}
\toprule
run & cert. & incumbent / gap & columns & nodes & tree s & work & wall s & GB\\
\midrule
branch-and-check, 3600 s & \val{ejor16_tl3600_certified}/\val{ejor16_tl3600_runs} & gap \val{ejor16_tl3600_gap_pct}\,\% & & \val{ejor16_tl3600_nodes} & & & \val{ejor16_tl3600_wall_s} & \\
branch-and-check, 10\,800 s & \val{ejor16_tl10800_certified}/\val{ejor16_tl10800_runs} & gap \val{ejor16_tl10800_gap_pct}\,\% & & \val{ejor16_tl10800_nodes} & & & \val{ejor16_tl10800_wall_s} & \\
\quad Xeon node & \val{ejor16up_plain_certified}/\val{ejor16up_plain_runs} & gap \val{ejor16up_plain_solver_gap_pct}\,\% & & \val{ejor16up_plain_nodes} & & & \val{ejor16up_plain_wall_s} & \val{ejor16up_plain_maxrss_gb}\\
\quad Xeon node, optimum as cutoff & \val{ejor16up_cutoff_certified}/\val{ejor16up_cutoff_runs} & gap \val{ejor16up_cutoff_solver_gap_pct}\,\% & & \val{ejor16up_cutoff_nodes} & & & \val{ejor16up_cutoff_wall_s} & \val{ejor16up_cutoff_maxrss_gb}\\
\midrule
cold start, dive & \val{cor2_n16_full_certified}/\val{cor2_n16_full_runs} & \val{honest_ref_cor2full_inc_gap_pct} & \val{cor2_n16_full_kept} & \val{cor2_n16_full_nodes} & \val{cor2_n16_full_tree_s} & \val{honest_ref_cor2full_work} & \val{cor2_n16_full_wall_s} & \val{cor2_n16_full_maxrss_gb}\\
\midrule
single job, heuristic & \val{honest_n16_uniform_certified}/\val{honest_n16_uniform_runs} & 0, \val{honest_n16_uniform_starts_at_exact}/\val{honest_n16_uniform_starts} & \val{honest_n16_uniform_kept_final} & \val{honest_n16_uniform_nodes} & \val{honest_n16_uniform_tree_s} & \val{honest_n16_uniform_work} & \val{honest_n16_uniform_wall_min_s}--\val{honest_n16_uniform_wall_max_s} & \val{honest_n16_uniform_maxrss_gb}\\
\quad the same, gap zero & \val{honest_n16_uniform_gap0_certified}/\val{honest_n16_uniform_gap0_runs} & 0, \val{honest_n16_uniform_gap0_starts_at_exact}/\val{honest_n16_uniform_gap0_starts} & \val{honest_n16_uniform_gap0_kept_final} & \val{honest_n16_uniform_gap0_nodes} & \val{honest_n16_uniform_gap0_tree_s} & \val{honest_n16_uniform_gap0_work} & \val{honest_n16_uniform_gap0_wall_min_s}--\val{honest_n16_uniform_gap0_wall_max_s} & \val{honest_n16_uniform_gap0_maxrss_gb}\\
single job, rose 0.4/0.3/0.2/0.1 & \val{honest_n16_p40_certified}/\val{honest_n16_p40_runs} & 0, \val{honest_n16_p40_starts_at_exact}/\val{honest_n16_p40_starts} & \val{honest_n16_p40_kept_final} & \val{honest_n16_p40_nodes} & \val{honest_n16_p40_tree_s} & \val{honest_n16_p40_work} & \val{honest_n16_p40_wall_min_s}--\val{honest_n16_p40_wall_max_s} & \val{honest_n16_p40_maxrss_gb}\\
single job, rose 0.6/0.25/0.1/0.05 & \val{honest_n16_p60_certified}/\val{honest_n16_p60_runs} & 0, \val{honest_n16_p60_starts_at_exact}/\val{honest_n16_p60_starts} & \val{honest_n16_p60_kept_final} & \val{honest_n16_p60_nodes} & \val{honest_n16_p60_tree_s} & \val{honest_n16_p60_work} & \val{honest_n16_p60_wall_min_s}--\val{honest_n16_p60_wall_max_s} & \val{honest_n16_p60_maxrss_gb}\\
\bottomrule
\end{tabular}
\end{table}

\subsection{Attribution}\label{sec:exp:attr}

Table~\ref{tab:attr} adds the levers one at a time on the $14\times14$ grid, five seeds on
the computers of the ladder, from the dive incumbent. The arm \texttt{base} is the fixing at the root on the
dense master, \texttt{loss} adds the loss coordinates, and \texttt{full} the propagation
between tiles and the refixing. The loss coordinates save root time and memory and not the tree (the
columns they keep differ from those of \texttt{base} because the reduced costs come from
another optimal basis of the same degenerate LP). The propagation shortens the tree, which falls
from \val{cor2_n14_loss_tree_s} to
\val{cor2_n14_full_tree_s} s on \val{cor2_n14_full_kept} columns instead of
\val{cor2_n14_loss_kept}. At $16\times16$, with 10\,800 s for the tree,
\texttt{base} certifies \val{cor2_n16_base_certified} of \val{cor2_n16_base_runs} seeds with
its own dive at \val{cor2_n16_base_inc} kW. It certifies still
\val{cor2_n16_baseinc_certified} of \val{cor2_n16_baseinc_runs} when it is given the
incumbent of \texttt{full} (\val{cor2_n16_baseinc_inc} kW), with \val{cor2_n16_baseinc_kept} columns in the tree and a bound
of \val{cor2_n16_baseinc_bound} at the limit. The \texttt{full} arm certifies
\val{cor2_n16_full_certified} of \val{cor2_n16_full_runs}. At the same incumbent, then,
\texttt{full} closes where \texttt{base} does not; this comparison does not separate the
loss coordinates, the propagation and the refixing, which the arms of the $14\times14$ grid
do, where the refixing never fires. Finally, we compared the Benders cut against the upwind cut with the $U$ of
Lemma~\ref{lem:support}, on the design V3+H3 without fixing and five seeds. The median tree
takes \val{cor2_n10_cut_tangent_tree_s} against \val{cor2_n10_cut_upwind_tree_s} s at
$10\times10$ and \val{cor2_n12_cut_tangent_tree_s} against \val{cor2_n12_cut_upwind_tree_s}
s at $12\times12$, with \val{cor2_n12_cut_tangent_nodes} against
\val{cor2_n12_cut_upwind_nodes} nodes: parity in time, fewer nodes.

\begin{table}[t]
\centering\small
\caption{The levers on the $14\times14$ grid, five seeds on the computers of the ladder, dive incumbent, medians:
wall and root seconds, peak memory (GB), columns in the tree, tree seconds and nodes.
\texttt{base}: fixing at the root; \texttt{loss}: plus the loss coordinates; \texttt{full}:
plus propagation between tiles and refixing.}
\label{tab:attr}
\begin{tabular}{l rrrrrr}
\toprule
arm & wall s & root s & GB & columns & tree s & nodes\\
\midrule
\texttt{base} & \val{cor2_n14_base_wall_s} & \val{cor2_n14_base_root_s} & \val{cor2_n14_base_maxrss_gb} & \val{cor2_n14_base_kept} & \val{cor2_n14_base_tree_s} & \val{cor2_n14_base_nodes}\\
\texttt{loss} & \val{cor2_n14_loss_wall_s} & \val{cor2_n14_loss_root_s} & \val{cor2_n14_loss_maxrss_gb} & \val{cor2_n14_loss_kept} & \val{cor2_n14_loss_tree_s} & \val{cor2_n14_loss_nodes}\\
\texttt{full} & \val{cor2_n14_full_wall_s} & \val{cor2_n14_full_root_s} & \val{cor2_n14_full_maxrss_gb} & \val{cor2_n14_full_kept} & \val{cor2_n14_full_tree_s} & \val{cor2_n14_full_nodes}\\
\bottomrule
\end{tabular}
\end{table}

\subsection{A pre-specified test set of unseen instances}\label{sec:exp:panel}

The instances of the ladder are those on which the pipeline was developed. To measure what
it does on instances it has not seen, we generated a test set of \val{panel_instances}
instances from a seed fixed before any run, and ran both codes on it once. The test set varies
three things the ladder keeps fixed. Each grid of $10\times10$, $12\times12$ and
$14\times14$ receives three masks of \emph{forbidden} cells, where no turbine may stand
while the wind still propagates, a feature of the EndoWake code.
The masks are a block, a connected region of 15\,\% of the cells; a corridor, the cells
within one cell of a random line; and scattered cells, 10\,\% of the cells drawn at random
(\val{panel_forbidden_min} to \val{panel_forbidden_max} forbidden cells per instance). For every grid and mask,
we consider two budgets, $\Nmax = n - 2$ and $n + 2$, the second being that of the ladder, and
two roses, the uniform one and the rose of weights $0.4, 0.3, 0.2, 0.1$ on the four
directions. The \val{panel_instances} instances are thus the full product of three grids,
three masks, two budgets and two roses. Every instance runs on
three seeds per code, with the limit of 3600 s. The pipeline is the design of the ladder,
unchanged, with the LP dive as incumbent; the branch-and-check is the frozen code, extended
by the few lines that read the forbidden cells and the direction weights, and nothing else.
The pipeline runs on the computers of the ladder (\texttt{\val{panel_full_host}}), the
branch-and-check on a second partition of the cluster of the ladder (\texttt{\val{panel_ejor_host}}:
Intel Xeon E5-2623 v3 at 3.00 GHz, four cores, 15 GB, one job per computer), two CPUs of
similar clock, each disclosed with its arm. The counts below are outcomes under these two
configurations, within a limit of 3600 s on the tree of each code.

Table~\ref{tab:panel} reports the test set per grid and code; Appendix~\ref{app:panel} breaks it down by
budget, rose and mask. The pipeline certifies \val{panel_full_certified} of its
\val{panel_full_runs} runs and every instance on at least one seed. Its
\val{panel_full_censored} runs at the limit are at $14\times14$ and end with a gap of
\val{panel_full_censored_gap_max_pct}\,\% at most; every run, censored or not, ends with
a layout within \val{panel_full_below_best_max_pct}\,\% of the best of its instance.
The root gap of the pipeline stays below \val{panel_full_root_gap_pct_max}\,\% on every run,
and \val{panel_full_kept_final_pct}\,\% of the patterns reach the tree in the median. We
checked every final layout of both codes against its mask and its budget. The tight budget
is the hard case for the pipeline: at $\Nmax = n - 2$ the tree takes
\val{panel_full_bm2_nodes} nodes in the median, against \val{panel_full_bp2_nodes} at
$n + 2$, where the root closes most instances.
The branch-and-check certifies \val{panel_ejor_certified} of its \val{panel_ejor_runs} runs:
all \val{panel_ejor_n10_runs} at $10\times10$, in \val{panel_ejor_n10_wall_s} s in the median
against \val{panel_full_n10_wall_s} s for the pipeline, \val{panel_ejor_n12_certified} at
$12\times12$, and none at $14\times14$, where the pipeline certifies
\val{panel_full_n14_certified} of \val{panel_full_n14_runs} in \val{panel_full_n14_wall_s} s
in the median. Its root gap is \val{panel_ejor_root_gap_pct}\,\% in the median
(\val{panel_ejor_root_gap_pct_max}\,\% at most), against \val{panel_full_root_gap_pct}\,\%
for the pipeline; its runs at the limit end with a gap of up to
\val{panel_ejor_censored_gap_max_pct}\,\% and a layout up to
\val{panel_ejor_below_best_max_pct}\,\% below the best of their instance. On the
\val{panel_both_certified} pairs where both codes certify the same instance and seed, the two
codes return the same exact energy (relative difference \val{panel_agree_max_pct}\,\%), on
\val{panel_pairs_same_layout} of them the same layout and on \val{panel_pairs_other_layout}
two layouts of equal energy. On \val{panel_full_only} pairs only the pipeline certifies,
on \val{panel_ejor_only} only the branch-and-check, and on \val{panel_neither} neither.

\begin{table}[t]
\centering\small
\caption{The pre-specified test set of Section~\ref{sec:exp:panel}: \val{panel_instances}
unseen instances (three grids, three masks of forbidden cells, two budgets, two roses),
three seeds per instance and code, 3600 s on the tree of each code. Per grid and code: certified runs over
runs; the root bound of each code against the best exact energy found on the instance by
either code, in percent (median and, in parentheses, maximum over the runs); the columns
that reach the tree after the fixing and the propagation, as a share of the patterns
(median); nodes and wall seconds (medians); the gap at the end, bound against exact energy,
in percent (median); the final
layout of a run against the best of its instance (maximum); every maximum is rounded
outward. The pipeline runs on the computers
of the ladder, the branch-and-check on a second partition of the cluster of the ladder.}
\label{tab:panel}
\footnotesize\setlength{\tabcolsep}{4pt}
\begin{tabular}{l l r r r r r r r}
\toprule
grid & code & cert. & root gap \% & kept \% & nodes & wall s & gap \% & below best \%\\
\midrule
$10\times10$ & pipeline & \val{panel_full_n10_certified}/\val{panel_full_n10_runs} & \val{panel_full_n10_root_gap_pct} (\val{panel_full_n10_root_gap_pct_max}) & \val{panel_full_n10_kept_final_pct} & \val{panel_full_n10_nodes} & \val{panel_full_n10_wall_s} & \val{panel_full_n10_gap_pct} & \val{panel_full_n10_below_best_max_pct}\\
 & branch-and-check & \val{panel_ejor_n10_certified}/\val{panel_ejor_n10_runs} & \val{panel_ejor_n10_root_gap_pct} (\val{panel_ejor_n10_root_gap_pct_max}) & -- & \val{panel_ejor_n10_nodes} & \val{panel_ejor_n10_wall_s} & \val{panel_ejor_n10_gap_pct} & \val{panel_ejor_n10_below_best_max_pct}\\
$12\times12$ & pipeline & \val{panel_full_n12_certified}/\val{panel_full_n12_runs} & \val{panel_full_n12_root_gap_pct} (\val{panel_full_n12_root_gap_pct_max}) & \val{panel_full_n12_kept_final_pct} & \val{panel_full_n12_nodes} & \val{panel_full_n12_wall_s} & \val{panel_full_n12_gap_pct} & \val{panel_full_n12_below_best_max_pct}\\
 & branch-and-check & \val{panel_ejor_n12_certified}/\val{panel_ejor_n12_runs} & \val{panel_ejor_n12_root_gap_pct} (\val{panel_ejor_n12_root_gap_pct_max}) & -- & \val{panel_ejor_n12_nodes} & \val{panel_ejor_n12_wall_s} & \val{panel_ejor_n12_gap_pct} & \val{panel_ejor_n12_below_best_max_pct}\\
$14\times14$ & pipeline & \val{panel_full_n14_certified}/\val{panel_full_n14_runs} & \val{panel_full_n14_root_gap_pct} (\val{panel_full_n14_root_gap_pct_max}) & \val{panel_full_n14_kept_final_pct} & \val{panel_full_n14_nodes} & \val{panel_full_n14_wall_s} & \val{panel_full_n14_gap_pct} & \val{panel_full_n14_below_best_max_pct}\\
 & branch-and-check & \val{panel_ejor_n14_certified}/\val{panel_ejor_n14_runs} & \val{panel_ejor_n14_root_gap_pct} (\val{panel_ejor_n14_root_gap_pct_max}) & -- & \val{panel_ejor_n14_nodes} & \val{panel_ejor_n14_wall_s} & \val{panel_ejor_n14_gap_pct} & \val{panel_ejor_n14_below_best_max_pct}\\
\bottomrule
\end{tabular}
\end{table}

\paragraph{Reproducibility.} Every number of this section is read from the archives of the
campaigns by one script per campaign (the Slurm logs, the JSON of each run, the solver logs
and the memory from \texttt{sacct}). The incumbent of every run is the run's own, and no
layout, root or verdict is carried from one run to another. The two exceptions, both by
construction, are the \texttt{baseinc} arm, which receives the incumbent of \texttt{full},
and the cutoff runs of the branch-and-check at $16\times16$, which receive the optimum as
cutoff. The code of the
pipeline, the archives of the campaigns and the scripts that read them are deposited on
Zenodo~\citep{fischetti2026package}.

\section{Scope and limitations}\label{sec:limits}

We now discuss how far the approach extends beyond the setting of Section~\ref{sec:exp}, and where it stops.
The validity of the tile inequalities rests on two properties only: the restriction of a
feasible layout to a tile is one of its patterns, and turbines outside a tile can only lower
the power of its cells. The proof of Proposition~\ref{prop:columns}(a) therefore carries
over to any model whose cell rewards are non-decreasing in a field that is a variable of the
model and antitone in the layout. Exactness needs more, an optimality cut that is tight at
integer layouts: the cube coordinate supplies it for the unshifted cubic program, and
antitonicity alone does not give the concavity on which it rests. How tight the tiles are,
and what their enumeration costs, has to be measured model by model.

\paragraph{Roses.} The pipeline does not exploit the symmetries of the uniform rose: no symmetry device, no
stored layout, and the tables of the tiles are built per instance. What it takes from the
ladder is the axial geometry (Tables~\ref{tab:rose} and~\ref{tab:16}). Both families have
one free-stream speed. Below rated, a speed class changes only a factor. The inequalities
\eqref{eq:flow}--\eqref{eq:wake} are homogeneous in $(V, w)$, so the field at speed $V_k$ is
$V_k/V$ times the field at $V$, and the power of every cell scales by $(V_k/V)^3$, for every
layout. This holds as long as no installed cell falls below cut-in and the calibration of
the stencil is the same at every speed. A rose with speed classes is then the single-speed rose
with direction weights proportional to $\sum_k p_{s,k} V_k^3$, which the non-uniform roses
of Section~\ref{sec:exp} cover. Above rated the plateau binds, and the classes are
scenarios we have not measured.

The oblique wake is where the tiles stop being tight. Their inequalities stay valid, but the
tiles are axial strips, and on a rose of twelve directions the wake of a turbine crosses a strip in a few cells. The pattern powers are
then far from the field at the root, and the fixing removes less. Table~\ref{tab:d12} measures
it on the pipeline of the ladder, with the eight oblique directions propagated plane by plane
along the dominant axis, the EndoWake stencil drifted and stretched
to the direction, which the EndoWake paper leaves uncalibrated. At $10\times10$ the root gap goes
from \val{d12_ref4_n10_root_gap_pct} to \val{d12_n10_root_gap_pct}\,\%, the fixing keeps
\val{d12_n10_kept_pct}\,\% of the columns instead of \val{d12_ref4_n10_kept_pct}\,\%, and
the tree still closes the grid, in \val{d12_n10_tree_s} s and \val{d12_n10_nodes} nodes
on every seed. From $12\times12$ on the fixing keeps every column: the
$12\times12$ stops at 10\,800 s with a gap of \val{d12_n12_gap_pct}\,\%, and the $14\times14$
barely leaves the root. The tight root bound of the axial case is a property of the axial
propagation, where a pattern of a strip captures the dominant along-wind interactions. The oblique
propagation as implemented relaxes it by one to two orders of magnitude, and the table
measures the pipeline on it, not the physics. Tiles along every direction, strips on each of the six axes glued
through the same $x$, recover the bound at the optimum and not at the root. Six families of
overlapping blocks make a loose relaxation where two make a tight one, and the LP becomes
the cost. The classical benchmark of \citet{mosetti1994} is of this kind: a $10\times10$ grid,
with a pairwise wake model and, in two of its three cases, a rose of 36 directions.
\citet{cazzaro2022vns} use it in this journal with a fixed number of turbines.
Our working hypothesis is that the fine rose is the open case and not a limit of the device: the
columns are the right object there too. What is missing is not the rotation of the
strips, which we tried, but a coupling between the families stronger than the shared site
variables. The consistency equations of \citet{clausen2022} on the overlaps are the natural
candidate, and the LP their cost: the code has a two-cell form of them as an option, the
pipeline does not use it, and whether the full equations close the oblique gap at an
affordable root is open.

\begin{table}[t]
\centering\small
\caption{The rose of twelve equispaced directions on the pipeline of Table~\ref{tab:ladder}
(dive incumbent; 3600 s for the tree at $10\times10$, 10\,800 s above; five seeds, three at
$14\times14$), against the four axial directions on the same grids. The oblique wake is
propagated by the reference code as implemented, uncalibrated. Root gap against the certified
optimum at $10\times10$ and the best layout found above; columns kept by the fixing and after
the propagation, as a share of the patterns; nodes, tree seconds and the gap at the end,
bound against the exact energy of the final layout (medians); peak memory. The $14\times14$ rows at twelve directions run on a shared Xeon
node, the others on the computers of the ladder.}
\label{tab:d12}
\footnotesize\setlength{\tabcolsep}{4pt}
\begin{tabular}{l r r r r r r r r r}
\toprule
grid & dirs & cert. & root gap \% & kept \% & propagated \% & nodes & tree s & gap \% & GB\\
\midrule
$10\times10$ & 4 & \val{d12_ref4_n10_certified}/\val{d12_ref4_n10_runs} & \val{d12_ref4_n10_root_gap_pct} & \val{d12_ref4_n10_kept_pct} & \val{d12_ref4_n10_kept_final_pct} & \val{d12_ref4_n10_nodes} & \val{d12_ref4_n10_tree_s} & \val{d12_ref4_n10_gap_pct} & \val{d12_ref4_n10_maxrss_gb}\\
 & 12 & \val{d12_n10_certified}/\val{d12_n10_runs} & \val{d12_n10_root_gap_pct} & \val{d12_n10_kept_pct} & \val{d12_n10_kept_final_pct} & \val{d12_n10_nodes} & \val{d12_n10_tree_s} & \val{d12_n10_gap_pct} & \val{d12_n10_maxrss_gb}\\
$12\times12$ & 4 & \val{d12_ref4_n12_certified}/\val{d12_ref4_n12_runs} & \val{d12_ref4_n12_root_gap_pct} & \val{d12_ref4_n12_kept_pct} & \val{d12_ref4_n12_kept_final_pct} & \val{d12_ref4_n12_nodes} & \val{d12_ref4_n12_tree_s} & \val{d12_ref4_n12_gap_pct} & \val{d12_ref4_n12_maxrss_gb}\\
 & 12 & \val{d12_n12_certified}/\val{d12_n12_runs} & \val{d12_n12_root_gap_pct} & \val{d12_n12_kept_pct} & \val{d12_n12_kept_final_pct} & \val{d12_n12_nodes} & \val{d12_n12_tree_s} & \val{d12_n12_gap_pct} & \val{d12_n12_maxrss_gb}\\
$14\times14$ & 4 & \val{d12_ref4_n14_certified}/\val{d12_ref4_n14_runs} & \val{d12_ref4_n14_root_gap_pct} & \val{d12_ref4_n14_kept_pct} & \val{d12_ref4_n14_kept_final_pct} & \val{d12_ref4_n14_nodes} & \val{d12_ref4_n14_tree_s} & \val{d12_ref4_n14_gap_pct} & \val{d12_ref4_n14_maxrss_gb}\\
 & 12 & \val{d12_n14_certified}/\val{d12_n14_runs} & \val{d12_n14_root_gap_pct} & \val{d12_n14_kept_pct} & \val{d12_n14_kept_final_pct} & \val{d12_n14_nodes} & \val{d12_n14_tree_s} & \val{d12_n14_gap_pct} & \val{d12_n14_maxrss_gb}\\
\bottomrule
\end{tabular}
\end{table}

\paragraph{The next grid.} At $17\times17$ the single job, given 512 GB, leaves the grid
open. The heuristic reaches \val{honest_n17_uniform_heur_best} kW on three seeds, in
\val{honest_n17_uniform_heur_s} s and \val{honest_n17_uniform_starts_at_exact} of
\val{honest_n17_uniform_starts} starts. The explicit root has \val{honest_n17_uniform_patterns}
patterns against \val{honest_n16_uniform_patterns} at $16\times16$, takes
\val{honest_n17_uniform_root_s} s and \val{honest_n17_uniform_maxrss_gb} GB, and its gap is
\val{honest_n17_uniform_root_gap_pct}\,\% against \val{honest_n16_uniform_root_gap_pct}. After
10\,800 s the tree holds \val{honest_n17_uniform_kept_final} columns on every seed, and the
best bound leaves \val{honest_n17_uniform_gap_min_pct}\,\%. A longer run
from the same incumbent, on shared Xeon nodes, stops at the job limit after
\val{tiles17inc_v4s2h4s2_tree_s_min}--\val{tiles17inc_v4s2h4s2_tree_s_max} s of tree. It
has explored \val{tiles17inc_v4s2h4s2_nodes_min}--\val{tiles17inc_v4s2h4s2_nodes_max} nodes,
\val{tiles17inc_v4s2h4s2_open_nodes_min}--\val{tiles17inc_v4s2h4s2_open_nodes_max} are
open, and the gap is
\val{tiles17inc_v4s2h4s2_gap_min_pct}--\val{tiles17inc_v4s2h4s2_gap_max_pct}\,\%. The grid is
one step past the pipeline, in memory and time.

\paragraph{Memory.} Memory can become an issue. The columns of the
$16\times16$ grid are enumerated and stored, and the root LP prices all of them. The job
needs a node with the memory of Table~\ref{tab:16}, and \val{cor2_n16_base_maxrss_gb} GB
without the loss coordinates (arm \texttt{base} of Table~\ref{tab:attr}, which does not close). On a computer of the ladder it runs out of memory at \val{cor2_probe_5026944_0_maxrss_gb} GB. Beyond this grid the memory grows with the
patterns, that is, with the width of the strips and with their length
(\val{honest_n17_uniform_patterns} patterns at $17\times17$ against
\val{honest_n16_uniform_patterns}); the cap $\Nmax$ never binds inside a strip
of the design. A root generated by column generation is our first alternative.

\paragraph{Scope.} The paper is a certification tool for the model, not a design tool for
a real-world site. The grid, the single speed and the cubic curve below rated are those of
\citet{fischetti2026endowake}; the forbidden cells of Section~\ref{sec:exp:panel} are as
far as the device goes toward a site; and the device certifies what the model says---errors
included. The useful target, in our view, is a certificate within a tenth of a percent of a
heuristic layout on a finer rose.

\section{Conclusions}\label{sec:concl}

We have shown that the EndoWake model can benefit from a Dantzig--Wolfe reformulation on tiles. The
layout of a tile is a convex combination of its patterns, enumerated once with their
powers. Antitonicity makes the pattern powers valid bounds on the cells, and the Benders cut
of the cube coordinate restores exactness at integer layouts in the unshifted cubic program; the
implementation carries a shift that keeps every bound valid, with a residual measured in
Section~\ref{sec:exp:setup}. We have then shown that the
columns need not be branched on. A good incumbent, the reduced costs of the root and a
propagation between overlapping tiles shrink the master to a size the tree closes in one
job. The incumbent is the optimum from the tile heuristic on the GPU, or a layout
\val{honest_ref_cor2full_inc_gap_pct}\,\% under it from the LP dive of the cold start at
$16\times16$. This holds on the grids up to $16\times16$ of the EndoWake ladder
and on two non-uniform roses at $16\times16$. The
branch-and-check leaves the $16\times16$ grid open, also on the
node of our cold start and with the optimum as cutoff. On six
non-uniform roses up to $14\times14$ the fixing at the root alone suffices. On a
pre-specified test set of \val{panel_instances} unseen instances with forbidden cells, two
budgets and a non-uniform rose, the pipeline certifies every instance up to $14\times14$.
The branch-and-check certifies the $10\times10$ instances of that test set,
\val{panel_ejor_n12_certified} runs at $12\times12$ and none at $14\times14$.
What the device
does not yet make tight is the oblique wake, where its inequalities stay valid but loose.
What it does not yet afford is the next grid:
at $17\times17$ the explicit root takes \val{honest_n17_uniform_maxrss_gb} GB and three
hours of tree leave \val{honest_n17_uniform_gap_min_pct}\,\% open. Future research should
therefore address the fine rose, where strips along the directions need a coupling stronger
than the shared site variables, and the next grid, where the columns should be priced in
the tree rather than enumerated at the root.

\appendix
\section{Validity of the reductions}\label{app:proofs}

The pipeline applies the argument of Section~\ref{sec:shrink} in the tree, with a moving incumbent.

\begin{proposition}\label{prop:valid}
Let the master carry the columns \eqref{eq:columns} of a family of tiles and any set of
inequalities valid at every feasible layout, let $v$ be the value of its root LP, and let
$\bar c$ be the reduced costs at an optimal dual solution of that LP, with the bounds of
the variables counted among the constraints so that $\bar c \le 0$. Let $L$ be any value,
and call a column $\lambda^R_T$ \emph{dead} for $L$ if $v + \bar c^R_T < L$. Then:
(a) no layout of energy above $L$ uses a dead column, at the root as in every node of the
tree, and a column dead for $L$ is dead for every $L' \ge L$;
(b) a layout of energy above $L$ uses only columns that survive the propagation between
tiles;
(c) every layout of energy above $L$ is, with its exact powers, a feasible point of the
master reduced to the surviving columns; hence a bound $U$ on the reduced master bounds
every layout of energy above $L$, and if the reduced master has no solution above $L$, no
layout has energy above $L$.
\end{proposition}

\begin{proof}
(a) Let $X$ be a feasible layout and $p$ the point of the master with $x$ the indicator of
$X$, $z$ its exact powers and $\lambda$ the indicators of its patterns: $p$ is feasible for
the root LP by Proposition~\ref{prop:columns}(a) and the validity of the inequalities, and
its objective is $E(X)$. With $y$ an optimal dual solution, LP duality gives
$E(X) \le v + \bar c^\top p$, and $\bar c \le 0$ with $p \ge 0$ gives
$\bar c^\top p \le \bar c^R_T$ for every column with $\lambda^R_T = 1$; hence $E(X) > L$
forces $v + \bar c^R_T > L$ on every column that $X$ uses. The bound uses the root dual
only, so it holds in every node, whose branching only restricts the feasible points; the
last claim is immediate.
(b) By (a) the patterns $X \cap R$ of a layout above $L$ all survive the fixing; for every
pair of overlapping tiles the patterns $X \cap R$ and $X \cap R'$ agree on $R \cap R'$, so
each is kept by every pass of the propagation, which removes a pattern only when no
surviving pattern of a neighbor agrees with it.
(c) By (b) every layout above $L$, with its exact powers, is a feasible point of the
reduced master, which is a restriction of the master; its energy is then at most any valid
bound on the reduced master, and both claims follow.
\end{proof}

A tree that closes on the reduced master with the incumbent entered at its exact energy $L$
proves that the reduced master has no solution above its final bound $U \ge L$, hence, by
(c), that the optimum of the model lies in $[L, U]$. The master of the implementation is a
relaxation of the model (Section~\ref{sec:tiles:col}): its shifted tables and cuts and its
unshifted chords are valid at every layout, which is all that (a)--(c) use. At gap zero $U$
is the value $M$ that the master credits to the final layout, which exceeds $L$ by the
residual of Section~\ref{sec:exp:setup}. The refixing in the tree is
(a) applied with the energy $L'$ of each new incumbent.

\section{The test set by budget, rose and mask}\label{app:panel}

Table~\ref{tab:panelgroups} breaks the test set of Section~\ref{sec:exp:panel} down by budget,
rose and mask, for each code; a budget or rose group holds half of the runs of a code, and a
mask group a third. For the pipeline, the tight budget $\Nmax = n - 2$ is the hard case and
$n + 2$, the budget of the ladder, the easy one; the rose of weights $0.4, 0.3, 0.2, 0.1$
costs more nodes than the uniform one; among the masks (block, corridor and scattered cells,
called \texttt{lake}, \texttt{road} and \texttt{holes} in the deposited code), the block is
the easiest, as the root closes most of its instances, and the scattered cells are the
hardest. The runs at the limit combine the scattered cells, the tight budget and the
non-uniform rose. For the branch-and-check the grid decides, not the group: its certified
runs are those at $10\times10$ and $12\times12$ (Table~\ref{tab:panel}). Its root gap is
largest at the loose budget, which is where the root gap of the pipeline is smallest.

\begin{table}[ht]
\centering\small
\caption{The test set by group and code (the columns of Table~\ref{tab:panel}): certified runs
over runs; the root bound of each code against the best exact energy of the instance, in
percent (median, maximum in parentheses); the columns that reach the tree of the pipeline as
a share of the patterns (median); nodes and wall seconds (medians); the gap at the end, bound
against exact energy (median); the final layout against the best of its instance (maximum);
every maximum is rounded outward.}
\label{tab:panelgroups}
\footnotesize\setlength{\tabcolsep}{4pt}
\begin{tabular}{l l r r r r r r r}
\toprule
group & code & cert. & root gap \% & kept \% & nodes & wall s & gap \% & below best \%\\
\midrule
$\Nmax = n - 2$ & pipeline & \val{panel_full_bm2_certified}/\val{panel_full_bm2_runs} & \val{panel_full_bm2_root_gap_pct} (\val{panel_full_bm2_root_gap_pct_max}) & \val{panel_full_bm2_kept_final_pct} & \val{panel_full_bm2_nodes} & \val{panel_full_bm2_wall_s} & \val{panel_full_bm2_gap_pct} & \val{panel_full_bm2_below_best_max_pct}\\
 & branch-and-check & \val{panel_ejor_bm2_certified}/\val{panel_ejor_bm2_runs} & \val{panel_ejor_bm2_root_gap_pct} (\val{panel_ejor_bm2_root_gap_pct_max}) & -- & \val{panel_ejor_bm2_nodes} & \val{panel_ejor_bm2_wall_s} & \val{panel_ejor_bm2_gap_pct} & \val{panel_ejor_bm2_below_best_max_pct}\\
$\Nmax = n + 2$ & pipeline & \val{panel_full_bp2_certified}/\val{panel_full_bp2_runs} & \val{panel_full_bp2_root_gap_pct} (\val{panel_full_bp2_root_gap_pct_max}) & \val{panel_full_bp2_kept_final_pct} & \val{panel_full_bp2_nodes} & \val{panel_full_bp2_wall_s} & \val{panel_full_bp2_gap_pct} & \val{panel_full_bp2_below_best_max_pct}\\
 & branch-and-check & \val{panel_ejor_bp2_certified}/\val{panel_ejor_bp2_runs} & \val{panel_ejor_bp2_root_gap_pct} (\val{panel_ejor_bp2_root_gap_pct_max}) & -- & \val{panel_ejor_bp2_nodes} & \val{panel_ejor_bp2_wall_s} & \val{panel_ejor_bp2_gap_pct} & \val{panel_ejor_bp2_below_best_max_pct}\\
uniform rose & pipeline & \val{panel_full_uniform_certified}/\val{panel_full_uniform_runs} & \val{panel_full_uniform_root_gap_pct} (\val{panel_full_uniform_root_gap_pct_max}) & \val{panel_full_uniform_kept_final_pct} & \val{panel_full_uniform_nodes} & \val{panel_full_uniform_wall_s} & \val{panel_full_uniform_gap_pct} & \val{panel_full_uniform_below_best_max_pct}\\
 & branch-and-check & \val{panel_ejor_uniform_certified}/\val{panel_ejor_uniform_runs} & \val{panel_ejor_uniform_root_gap_pct} (\val{panel_ejor_uniform_root_gap_pct_max}) & -- & \val{panel_ejor_uniform_nodes} & \val{panel_ejor_uniform_wall_s} & \val{panel_ejor_uniform_gap_pct} & \val{panel_ejor_uniform_below_best_max_pct}\\
rose 0.4--0.1 & pipeline & \val{panel_full_p40_certified}/\val{panel_full_p40_runs} & \val{panel_full_p40_root_gap_pct} (\val{panel_full_p40_root_gap_pct_max}) & \val{panel_full_p40_kept_final_pct} & \val{panel_full_p40_nodes} & \val{panel_full_p40_wall_s} & \val{panel_full_p40_gap_pct} & \val{panel_full_p40_below_best_max_pct}\\
 & branch-and-check & \val{panel_ejor_p40_certified}/\val{panel_ejor_p40_runs} & \val{panel_ejor_p40_root_gap_pct} (\val{panel_ejor_p40_root_gap_pct_max}) & -- & \val{panel_ejor_p40_nodes} & \val{panel_ejor_p40_wall_s} & \val{panel_ejor_p40_gap_pct} & \val{panel_ejor_p40_below_best_max_pct}\\
block & pipeline & \val{panel_full_lake_certified}/\val{panel_full_lake_runs} & \val{panel_full_lake_root_gap_pct} (\val{panel_full_lake_root_gap_pct_max}) & \val{panel_full_lake_kept_final_pct} & \val{panel_full_lake_nodes} & \val{panel_full_lake_wall_s} & \val{panel_full_lake_gap_pct} & \val{panel_full_lake_below_best_max_pct}\\
 & branch-and-check & \val{panel_ejor_lake_certified}/\val{panel_ejor_lake_runs} & \val{panel_ejor_lake_root_gap_pct} (\val{panel_ejor_lake_root_gap_pct_max}) & -- & \val{panel_ejor_lake_nodes} & \val{panel_ejor_lake_wall_s} & \val{panel_ejor_lake_gap_pct} & \val{panel_ejor_lake_below_best_max_pct}\\
corridor & pipeline & \val{panel_full_road_certified}/\val{panel_full_road_runs} & \val{panel_full_road_root_gap_pct} (\val{panel_full_road_root_gap_pct_max}) & \val{panel_full_road_kept_final_pct} & \val{panel_full_road_nodes} & \val{panel_full_road_wall_s} & \val{panel_full_road_gap_pct} & \val{panel_full_road_below_best_max_pct}\\
 & branch-and-check & \val{panel_ejor_road_certified}/\val{panel_ejor_road_runs} & \val{panel_ejor_road_root_gap_pct} (\val{panel_ejor_road_root_gap_pct_max}) & -- & \val{panel_ejor_road_nodes} & \val{panel_ejor_road_wall_s} & \val{panel_ejor_road_gap_pct} & \val{panel_ejor_road_below_best_max_pct}\\
scattered & pipeline & \val{panel_full_holes_certified}/\val{panel_full_holes_runs} & \val{panel_full_holes_root_gap_pct} (\val{panel_full_holes_root_gap_pct_max}) & \val{panel_full_holes_kept_final_pct} & \val{panel_full_holes_nodes} & \val{panel_full_holes_wall_s} & \val{panel_full_holes_gap_pct} & \val{panel_full_holes_below_best_max_pct}\\
 & branch-and-check & \val{panel_ejor_holes_certified}/\val{panel_ejor_holes_runs} & \val{panel_ejor_holes_root_gap_pct} (\val{panel_ejor_holes_root_gap_pct_max}) & -- & \val{panel_ejor_holes_nodes} & \val{panel_ejor_holes_wall_s} & \val{panel_ejor_holes_gap_pct} & \val{panel_ejor_holes_below_best_max_pct}\\
\bottomrule
\end{tabular}
\end{table}

\section*{Data availability}
The code of the pipeline, the archives of the campaigns of Section~\ref{sec:exp}, the
scripts that read them and write every number of this paper, and the generators of the
instances, of the ladder and of the test set, are deposited on
Zenodo~\citep{fischetti2026package}, \url{https://doi.org/10.5281/zenodo.23636074}. The instances of the ladder are those
of \citet{fischetti2026endowake}; the test set is generated by the deposited code.

\section*{Declaration of competing interest}
The authors declare that they have no known competing financial interests or personal
relationships that could have appeared to influence the work reported in this paper.

\section*{CRediT authorship contribution statement}
\textbf{Martina Fischetti}: Conceptualization, Data curation, Investigation, Methodology,
Validation, Writing---original draft.
\textbf{Matteo Fischetti}: Conceptualization, Formal analysis, Methodology, Software,
Writing---review \& editing.

\section*{Acknowledgments}
The first author was funded by the Spanish Ministry of Science and Innovation through
project PID2024-162616OA-I00 and by the Ramon y Cajal grant RYC2023-042623-I. The work of
the second author was partially supported by MiUR, Italy.

\section*{Declaration of generative AI and AI-assisted technologies in the writing process}
During the preparation of this work the authors used Claude (Anthropic) and Codex (OpenAI)
to assist in drafting and editing the manuscript text, in developing the code and the
scripts released with the paper, and to review successive drafts as a simulated referee.
After using these tools, the authors reviewed, verified, and edited
all content---including all mathematical statements and proofs, the input data, and the
computational results---and take full responsibility for the content of the published
article.

\end{document}